\documentclass{article}
\usepackage{fancyhdr}
\usepackage[british]{babel}

\usepackage{amssymb, amsmath, amsthm, amsfonts, enumerate}
\usepackage{amsmath,setspace,scalefnt}
\usepackage[usenames,dvipsnames,svgnames,table]{xcolor}
\usepackage{graphicx,tikz,caption,subcaption}
\usetikzlibrary{patterns}
\usetikzlibrary{shapes}
\usepackage{fullpage}
\usepackage{hyperref}
\usepackage{enumitem,verbatim}
\usepackage{multicol}
\usepackage{float}
\usepackage{cleveref}
\usepackage[normalem]{ulem}

\usepackage[margin=2.5cm]{geometry}

\definecolor{gray}{rgb}{0.25, 0.25, 0.25}

\newtheorem{theorem}{Theorem}[section]
\newtheorem{lemma}[theorem]{Lemma}
\newtheorem{cor}[theorem]{Corollary}

\newtheorem*{theorem*}{Theorem}

\theoremstyle{definition}

\newcounter{propcounter}
\newenvironment{proplist}[1][]{%
    \stepcounter{propcounter}%
    \enumerate[label = {\bfseries \Alph{propcounter}\arabic{enumi}}, #1]%
}
{\endenumerate}

\theoremstyle{plain}
\newtheorem{claim}[theorem]{Claim}

\newtheorem{prop}[theorem]{Proposition}
\theoremstyle{definition}

\theoremstyle{definition}

\theoremstyle{definition}

\theoremstyle{definition}
\newtheorem{defn}[theorem]{Definition}
\theoremstyle{definition}

\theoremstyle{definition}

\newcommand{\eps}{\varepsilon}

\title{Ramsey--Dirac theory for bounded degree hypertrees}
\author{
Jie Han
\thanks{School of Mathematics and Statistics and Center for Applied Mathematics, Beijing Institute of Technology, Email: \texttt{han.jie@bit.edu.cn}.
Partially supported by National Natural Science Foundation of China (12371341).}
\and
Seonghyuk Im
\thanks{Department of Mathematical Science, KAIST, South Korea. Email: {\tt $\{$seonghyuk, jaehoon.kim$\}$@kaist.ac.kr}. Supported by the National Research Foundation of Korea (NRF) grant funded by the Korean government(MSIT) No. RS-2023-00210430. 
}~\thanks{Extremal Combinatorics and Probability Group (ECOPRO), Institute for Basic Science (IBS). Supported by the Institute for Basic Science (IBS-R029-C4)}
\and Jaehoon Kim\footnotemark[2]
\and Donglei Yang 
\thanks{School of Mathematics, Shandong University, Email: {\tt dlyang@sdu.edu.cn}. Supported by Natural Science Foundation of China (12101365) and Natural Science Foundation of Shandong Province (ZR2021QA029).}
}
\date{\today}

\begin{document}

\maketitle
\begin{abstract}
Ramsey--Tur\'an theory considers Tur\'an type questions in Ramsey-context, asking for the existence of a small subgraph in a graph $G$ where the complement $\overline{G}$ lacks an appropriate subgraph $F$, such as a clique of linear size. 
Similarly, one can consider Dirac-type questions in  Ramsey context, asking for the existence of a spanning subgraph $H$ in a graph $G$ where the complement $\overline{G}$ lacks an appropriate subgraph $F$, which we call a Ramsey--Dirac theory question.

When $H$ is a connected spanning subgraph, the disjoint union $K_{n/2}\cup K_{n/2}$ of two large cliques shows that it is natural to consider complete bipartite graphs $F$. Indeed, Han, Hu, Ping, Wang, Wang and Yang in 2024 proved that if $G$ is an $n$-vertex graph with $\delta(G)=\Omega(n)$ where the complement $\overline{G}$ does not contain any complete bipartite graph $K_{m,m}$ with $m=\Omega(n)$, then $G$ contains every $n$-vertex bounded degree tree $T$ as a subgraph. 

Extending this result to the Ramsey--Dirac theory for hypertrees, we prove that if $G$ is an $n$-vertex $r$-uniform hypergraph with $\delta(G)=\Omega(n^{r-1})$ where the complement $\overline{G}$ does not contain any complete $r$-partite hypergraph $K_{m,m,\dots, m}$ with $m=\Omega(n)$, then $G$ contains every $n$-vertex bounded degree hypertree $T$ as a subgraph.
We also prove the existence of matchings and loose Hamilton cycles in the same setting, which extends the result of Mcdiarmid and Yolov into hypergraphs. 

The lack of $K_{m,m,\dots, m}$ in $\overline{G}$ can be viewed as a very weak form of pseudorandomness condition on the hypergraph $G$. Hence, our results have some interesting implications on pseudorandom hypgraphs.
For examples, they generalize the universality result on randomly perturbed graphs by B\"ottcher, Han, Kohayakawa, Montgomery, Parczyk and Person in 2019 into hypergraphs and 
also strengthen the results on quasirandom hypergraphs by
Lenz, Mubayi and Mycroft in 2016 and Lenz and Mubayi in 2016 into hypergraphs satisfying a much weaker pseudorandomness condition.
\end{abstract}

\section{Introduction}\label{sec:intro}
Dirac-type questions ask to find a desired spanning structure in a dense (hyper-)graph.
Such a name came from the classical result of Dirac~\cite{dirac1952some} which asserts that every graph with minimum degree at least $n/2$ contains a Hamilton cycle. 
Since this influential work by Dirac, the minimum degree condition that guarantees the existence has been extensively studied for various spanning subgraphs.
For example, Hajnal and Szemer\'{e}di~\cite{hajnal1970proof} proved that if $G$ has minimum degree at least $(1-\frac{1}{k})n$ and $n$ is divisible by $k$, then $G$ contains a $K_k$-factor and Koml\'{o}s, S\'{a}rk\"{o}zy, and Szemer\'{e}di~\cite{komlos1995proof} proved that every graph with minimum degree at least $(1/2+o(1))n$ contains every bounded degree spanning tree as a subgraph. 
They further extended this result to trees with maximum degree $O(n/\log n)$ in \cite{komlos2001spanning}.
Generalizing these works, B\"{o}ttcher, Schacht and Taraz~\cite{bottcher2009proof} proved the bandwidth theorem which yields asymptotic minimum degree conditions for containing any bounded degree spanning subgraph with small size separators.

All of these results are (asymptotically) tight and extremal examples are complete $k$-partite graphs with slightly unbalanced $k$-partition for some $k$.
These extremal examples are not typical graphs as they admit a very rigid structure. In particular, they contain large independent sets. Thus, it is natural to ask whether the minimum degree condition can be weakened if we know that the host graph $G$ does not admit such rigid structures, hence far from the extremal examples. Such a question was considered for the Tur\'an-type problem of finding small graphs, giving rise to Ramsey--Tur\'an theory. A similar question can also be considered for Dirac-type questions, and perhaps the answers to such questions should be called Ramsey--Dirac theory. 
Unlike Ramsey--Tur\'an theory, although very natural, Ramsey--Dirac questions were not considered until Balogh, Molla and Sharifzadeh~\cite{Balogh2016triangle}.
In \cite{Balogh2016triangle}, they considered a Ramsey--Dirac question for triangle factors and proved that if an $n$-vertex graph $G$ with $3\mid n$ has sublinear independence number $\alpha(G) = o(n)$ and $\delta(G)\geq (1/2+o(1))n$, then $G$ contains $K_3$-factor, which is significantly weaker bound than $\frac{2}{3}n$ which comes from the Hajnal-Szemer\'{e}di theorem. This result was further extended by Knierim and Su~\cite{Knierim2021Kr} to $K_k$-factor for every $k \geq 4$.

While complete multipartite graphs are the only extremal examples for containing $K_k$-factors, there are other extremal examples if we consider `connected' spanning structures. Indeed, the disjoint union of two copies of cliques $K_{n/2}$ forms an extremal example for Dirac's theorem on Hamilton cycle. 
While this extremal example also admits very rigid structures, it does not contain any independent set of size at least three. This shows that sublinear independence number is not the right condition to include `rigid structures' in the host graph when we want to guarantee a connected spanning structures.
Inspired from this extremal example, one can consider `bipartite hole' of size $m$ in graph $G$, which is a pair of disjoint sets $X$ and $Y$ where $G$ has no edges between two sets.
This notion of a bipartite hole was introduced by 
Mcdiarmid and Yolov~\cite{McDiarmid2017Hamilton} and they proved that for $\delta >0$, there exists $\alpha=\alpha(\delta)>0$ such that any $n$ vertex graph $G$ having no bipartite holes of size $\alpha n$ with $\delta(G)\geq \delta n$ contains a Hamilton cycle. 
Such results were proved for bounded degree spanning trees by Han, Hu, Ping, Wang, Wang,
and Yang~\cite{Han2024Spanning} and for
the powers of Hamilton cycle by Chen, Han, Tang and Yang~\cite{Chen2023powers}.
Unlike the subsequent results, McDiarmid and Yolov's result actually determined that the best possible value of $\alpha(\delta)$ is $\delta$.
Another notable result is by Dragani\'c, Munh\'a Correia and Sudakov~\cite{Draganic2024} who proved a pancyclicity theorem that generalizes the result McDiarmid and Yolov.

Dirac-type questions were considered not only for graphs but also for hypergraphs.
An $r$-uniform hypergraph $G$ (or simply an $r$-graph) is a collection of $r$-sets of vertices called edges.
For a set $S$ and an $r$-graph $G$, degree $d_G(S)$ of $S$ in $G$ is the number of edges of $G$ containing $S$ and the minimum $k$-codegree of $G$ is $\delta_k(G)=\min_{S\in \binom{V(G)}{k}} d_G(S)$.
R\"{o}dl, Ruci\'{n}ski, and Szemer\'{e}di~\cite{Rodl2009perfect} proved that any $n$-vertex $r$-graph with minimum $(r-1)$-codegree at least $(1 + o(1))\frac{n}{2}$ contains a ``tight Hamiltonian cycle'' when $n$ is sufficiently large. Inspired by this result, minimum (co-)degree conditions for the existence of various spanning structures were intensively studied in hypergraphs. See the following surveys~\cite{Kuhn2009embedding, Kuhn2014hamilton,Zhao2016recent} for more details. 

While Dirac-type questions have been extensively studied for hypergraphs as above, Ramsey--Dirac questions remain largely unexplored in this context. As far as we know, the only Ramsey--Dirac result on hypergraphs is by Han in \cite{Han2018} which determined the minimum $(k-1)$-degree condition guaranteeing perfect matchings in uniform hypergraphs with sublinear independence number. 
As perfect matching is far from being connected, sublinear independence number is the right parameter for this. 
In this paper, we prove the first Ramsey--Dirac result for connected stuctures in hypergraphs by considering loose Hamilton cycles and hypertrees. To properly discuss Ramsey--Dirac questions for connected hypergraphs, we first need to define a notion of $r$-partite hole in hypergraphs.

For an $r$-graph $G$ and vertex sets $X_1,\dots,X_r \subseteq V(G)$, let $e_G(X_1,\dots,X_r)$ be the number of tuples $(x_1,\dots,x_r)$ with $x_i \in X_i$ for each $i \in [r]$ and $\{x_1,\dots,x_r\} \in E(G)$. Note that an edge may be counted more than once in $e_G(X_1,\dots, X_r)$ if the sets $X_1,\dots, X_r$ are not pairwise disjoint.
We define \emph{the $r$-partite hole number} $\alpha^*(G)$ of $G$ to be  the largest $t$ such that there exist $X_1, \ldots, X_r \subseteq V(G)$ with $|X_i|=t$ and $e_G(X_1, \ldots, X_r)=0$. Those sets do not have to be pairwise disjoint as it is convenient for our discussion later. However, even if one insists that they are pairwise disjoint, it will make very little difference in our discussion because it will only make the parameter smaller by at most $r$ times. 
For a given tuple $\mathcal{V}=(V_1,\dots, V_r)$ of subsets of $V(G)$, we write $\alpha^*_{\mathcal{V}}(G)$ be the largest $t$ such that there exist $X_1\subseteq V_1, \dots, X_{r}\subseteq V_r$ with $|X_i|=t$ for all $i\in [r]$ and $e_G(X_1,\dots, X_r)=0$. 
A loose cycle is an $r$-graph with vertex set $v_1, \ldots, v_{(r-1)n}$ and edges $\{v_{(r-1)(i-1)+1}, v_{(r-1)(i-1)+2}, \ldots, v_{(r-1)i+1}\}$ for every $i \in [n]$ where $v_{(r-1)n+1}=v_1$.

With these notations, our main results can be stated as follows.

\begin{theorem}\label{thm:main_theorem_cycle}
    For every integer $r>0$ and a real $\varepsilon>0$, there exists a real $\alpha=\alpha(r,\varepsilon)>0$ such that the following holds for every $n$ divisible by $r-1$.
    If an $n$-vertex $r$-graph $G$ satisfies $\delta_1(G) \geq \varepsilon n^{r-1}$ and $\alpha^*(G) < \alpha n$, then $G$ contains loose Hamilton cycle.
\end{theorem}

A hypergraph $T$ is a \emph{linear hypertree} if $T$ is linear (i.e., for every two distinct vertices, there is at most one edge containing both of them) hypergraph and for every two vertices, there exists a unique path connecting two of them. For the formal definition of linear hypertree, see Section~\ref{sec:prelim_hypertree}.

\begin{theorem}\label{thm:main_theorem}
    For every integers $r, \Delta$ and a real $\varepsilon>0$, there exists a real $\alpha=\alpha(r,\Delta,\varepsilon)>0$ such that the following holds for every $n$ where $n-1$ is divisible by $r-1$.
    If an $n$-vertex $r$ graph $G$ satisfies $\delta_1(G) \geq \varepsilon n^{r-1}$ and $\alpha^*(G) < \alpha n$ (i.e., $G$ has no $r$-partite hole of size $\alpha n$), then $G$ contains every $n$-vertex linear hypertree $T$ with $\Delta(T)\leq \Delta$ as a subgraph.
\end{theorem}

In the course of proving the theorem, we also prove the theorem for perfect matchings.
\begin{theorem}\label{thm:perfect_matching}
    For every integer $r$ and a real $\varepsilon>0$, there exists a real $\alpha=\alpha(r,\varepsilon)>0$ such that the following holds for every $n$ divisible by $r$.
    If an $n$-vertex $r$ graph $G$ satisfies $\delta_1(G) \geq \varepsilon n^{r-1}$,  $\alpha^*(G) < \alpha n$ (i.e., $G$ has no $r$-partite hole of size $\alpha n$), then $G$ contains a perfect matching.
\end{theorem}

Another important line of research is about pseudorandom hypergraphs. Many types of of pseudorandomness conditions have been explored in conjuction with the existence of substructures. As there are too many important results to list in this direction, we just refer the reader to \cite{Krivelevich2006, Towsner} for examples of (though certainly not all) pseudorandomness conditions considered for graphs and hypergraphs.
Our $r$-partite hole condition can be viewed as a very weak pseudorandomness condition. Hence, it has implications on hypergraphs with certain pseudorandomness conditions.
For example, our results \Cref{thm:main_theorem_cycle} and \ref{thm:perfect_matching} imply the result of Lenz, Mubayi and Mycroft~\cite{Lenz2016Hamilton} and Lenz and Mubayi~\cite{Lenz2016Perfect} regarding the existence
of loose Hamilton cycles and perfect matchings in quasirandom hypergraphs which satisfy a $(d, \mu)$-dense condition. Here,  $(d,\mu)$-dense condition requires $e_G(X_1,\dots, X_r)\geq \Omega(n^{r})$ for all linear size sets $X_i$ while the $r$-partite hole condition on $\alpha^*(G)$ requires only $e_G(X_1,\dots, X_r)>0$.

\subsection{Randomly perturbed graph}

Graphs without rigid structures in Ramsey--Dirac theory are also closely related to the Dirac-type questions on randomly perturbed graphs. A randomly perturbed graph is the union of a given graph with a binomial random graph $G(n,p)$. This is a graph-theoretic analogue of the smoothed analysis of algorithm introduced by Speilman and Teng \cite{Spielman2004smoothness}. In \cite{Spielman2004smoothness}, they estimate the performance of the simplex algorithm in 'realistic' settings using randomly perturbed inputs and a combination of worst case and average case analysis.
Using the same philosophy, Bohman, Frieze and Marti~\cite{Bohman2003random} proved that for a given $\varepsilon>0$, there exists a constant $C=C(\varepsilon)>0$ such that for any graph $G$ with $\delta(G) \geq \varepsilon n$, the union $G \cup G(n, C/n)$ contains a Hamilton cycle.
There are numerous results on randomly perturbed graphs~\cite{Bedenknecht2019powers,Bohman2004adding, Chang2023powers, Han2020Hamiltonicity, McDowell2018Hamilton}, including the Dirac-type question on bounded degree spanning trees by Krivelevich, Kwan and Sudakov~\cite{Krivelevich2017bounded}.
They proved that for every $\varepsilon>0$ and integer $\Delta>0$, there exists $C=C(\varepsilon,\Delta)>0$ such that if $G$ is an $n$-vertex graph with $\delta(G) \geq \varepsilon n$, then $G \cup G(n, C/n)$ contains every bounded degree spanning tree $T$ with maximum degree at most $\Delta$ with high probability.
This was extended into trees with maximum degree $O(n/\log n)$ by Joos and Kim~\cite{Joos2020Spanning}.
As a random graph $G(n, C/n)$ has no bipartite hole of size $C^{-1/3}n$ for large $C$ with high probability, the Ramsey--Dirac type results by Mcdiarmid and Yolov~\cite{McDiarmid2017Hamilton} and by Han, Hu, Ping, Wang, Wang and Yang~\cite{Han2024Spanning} imply the corresponding results for randomly perturbed graphs. Moreover, they yield a stronger universality result.

In the same manner, our main result implies the following statement for randomly perturbed hypergraphs.
We write $G^{(r)}(n, p)$ to denote the random $r$-graph on $n$ vertices where each $r$-set is an edge independently with probability $p$.

\begin{cor}\label{thm:random_perturbation}
    For positive integers $r, \Delta$ and a real $\varepsilon>0$, there exists a real $C>0$ and an integer $n_0$ such that the following holds for every $n \geq n_0$.
    If $G$ is an $n$-vertex $r$-graph with $\delta_1(G) \geq \varepsilon n^{r-1}$, then the following holds with high probability.
    \begin{enumerate}
        \item If $(r-1) \mid (n-1)$, then $G \cup G^{(r)}(n, \frac{C}{n^{r-1}})$ contains
        all $n$-vertex linear hypertrees with maximum degree $\Delta$; 
        \item if $(r-1) \mid n$, then $G \cup G^{(r)}(n, \frac{C}{n^{r-1}})$ contains a loose Hamilton cycle; 
        \item if $r \mid n$, then then $G \cup G^{(r)}(n, \frac{C}{n^{r-1}})$ contains a perfect matching.
    \end{enumerate}
\end{cor}
In fact, this corollary yields a stronger universality result. In \cite{Krivelevich2017bounded}, Krivelevich, Kwan and Sudakov asked whether a randomly perturbed graph $G\cup G(n,C/n)$ contains all bounded degree trees simultaneously with high probability. 
In other words, whether $G\cup G(n,C/n)$  is universal for all bounded degree trees.
Indeed, this was affirmatively answered by B\"ottcher, Han, Kohayakawa, Montgomery, Parczyk and Person~\cite{bottcher19}. \Cref{thm:random_perturbation} confirms that the universality holds for bounded degree hypertrees. 

\subsection{Rainbow version}
Another recent trends on the Dirac-type questions are their rainbow generalizations. 
Consider a system $\mathcal{G} = (G_1, \ldots, G_m)$ of graphs on the common vertex set $V$. 
A graph $H$ on the vertex set $V$ is called a \emph{rainbow} (or transversal) subgraph of $\mathcal{G}$ if there exists an injective function $\varphi:e(H) \rightarrow [m]$ such that for every $e \in E(H)$, $e \in E(G_{\varphi(e)})$.
This concept of rainbow subgraphs generalize the concept of rainbow subgraphs in edge-colored graphs, and is analogous to the colorful objects considered in discrete geometry (see \cite{Barany1982colorful, Holmsen2008caratheodory, Kalai2005Helly}). 
Dirac-type question for rainbow subgraphs was first considered by Joos and Kim~\cite{Joos2020rainbow}. 
They proved that if $m=n$ and $\delta(G_i) \geq n/2$ for every $i \in [n]$, then $\mathcal{G}$ contains a rainbow Hamilton cycle. After their work, Dirac-type questions for rainbow structure has received a lot of attention in recent years~\cite{Chakraborti2023bandwidth, Cheng2023transversal,Gupta2023general, Lee2023highdimensional, Montgomery2022transversal, Tang2023rainbow}.
Indeed, our main result can be extended into rainbow version as follows.

\begin{theorem}\label{thm:main_rainbow_version}
    For every integers $r, \Delta$ and a real $\varepsilon>0$, there exists a real $\alpha>0$ such that the following holds for every $n$ with $(r-1)\mid (n-1)$.
    Let $m = \frac{n-1}{r-1}$ and $\mathcal{G}=(G_1, \ldots, G_m)$ be a system of $n$-vertex $r$-graphs on the common vertex set $V$.
    If $\delta_1(G_i) \geq \varepsilon n^{r-1}$ and $\alpha^*(G_i) < \alpha n$ for every $i \in [m]$, then every $n$-vertex tree $T$ with $\Delta(T) \leq \Delta$ is a rainbow subgraph of $\mathcal{G}$.
\end{theorem}

The rainbow generalizations of Theorem~\ref{thm:main_theorem_cycle} for loose Hamilton cycles and \ref{thm:perfect_matching} for perfect matchings also can be proved. As their proof and statements are analogous to the ones of Theorem~\ref{thm:main_rainbow_version}, we omit them.

In order to prove Theorem~\ref{thm:main_rainbow_version}, we in fact prove a ``bipartite'' version of the Theorem~\ref{thm:main_theorem}. With this bipartite version on our hand, we consider a system of $r$-graphs $(G_1, \ldots, G_m)$ as a ``bipartite'' $(r+1)$-graph by adding vertices corresponds to each colors to deduce \Cref{thm:main_rainbow_version} follows.

We now define what ``bipartite'' means here. An $r$-graph $G$ is \emph{$(x,y)$-graph} if there exists a vertex partition $X, Y$ of $V(G)$ such that for every $e \in E(G)$, $|e \cap X|=x$ and $|e \cap Y|=y$.
\begin{theorem}\label{thm:main_bipartite_version}
    For every integers $r, \Delta$ and a real $\varepsilon>0$, there exists a real $\alpha>0$ and such that the following holds for sufficiently large $n$ with $r\mid (n-1)$. 
    Let $G$ be an $n$-vertex $(r,1)$-graph with partition $X, Y$ where $|X| = \frac{(r-1)n+1}{r}$ and $|Y| = \tfrac{n-1}{r}$ and let $\mathcal{V}= (X,\dots, X, Y)$. 
    
    Suppose that $\delta_1(G) \geq \varepsilon n^r$ and $\alpha^*_{\mathcal{V}}(G) <\alpha n$, and
    suppose that $T$ is an $n$-vertex $(r+1)$-hypertree with partition $A, B$ such that all the vertices in $B$ has degree $1$ satisfying $\Delta(T)\leq \Delta$.
    Then $G$ contains a copy of $T$ such that all the vertices in $A$ mapped into $X$ and all the vertices in $B$ mapped into $Y$.
\end{theorem}

\begin{proof}[Proof of \Cref{thm:main_rainbow_version}]
    For given a system of $n$-vertex $r$-graphs $(G_1, \ldots, G_m)$ on $V$ and a tree $T$ with $\Delta(T) \leq \Delta$, we construct an $(r+1)$-graph $H$ on $V \cup C$ with $|C|=m $ and a linear $(r+1)$-hypertree $\hat{T}$ as follows.
    We label vertices of $C$ by $[m]$ and if an $r$-set $e$ is an edge of $G_i$, then we add an edge $e \cup \{i\}$.
    To construct $\hat{T}$, we add one additional vertex to each of the edge of $T$ and those newly added vertices will be in $B$. 
    Then we observe that $(G_1, \ldots, G_m)$ contains a rainbow copy of $T$ if and only if $H$ contains a copy of $\hat{T}$.
    By applying \Cref{thm:main_bipartite_version} to $H$, we conclude the proof.
\end{proof}

\section{Outline of the proof}
\label{sec:outline}
To prove \Cref{thm:main_theorem}, we develop the ideas introduced in \cite{Han2024Spanning} for the graph case.

Assume that we have a bounded degree $r$-uniform hypertree $T$ and an $r$-graph $G$ with $\delta_1(G) = \Omega(n^{r-1})$ and $\alpha^*(G)= o(n)$. 
An edge $e$ in $T$ is a leaf-edge if $r-1$ vertices in $e$ has degree one.
Consider the tree $T'$ obtained from $T$ by deleting all the leaf-edges and isolated vertices.
A leaf-edge $e$ of $T'$ is the center of a ``pendant star'' in the original tree $T$ which consists of $e$ and the leaf-edges of $T$ incident to $e$. A bare path (a linear path whose vertices other than two end vertices are not incident to any edges outside the path) $P$ in $T'$ together with the leaf-edges of $T$ adjacent to $P$ form a ``caterplillar'' in $T$. 
As $T'$ has either many leaf-edges or many disjoint bare paths, the original tree $T$ has either many non-trivial pendant stars or many disjoint caterpillars. Our proof deals with these two cases in different ways. See Section~\ref{sec:prelim_hypertree} for the formal definitions.

Let $V=V(G)$. We take $S_1,\dots, S_m$ to be either non-trivial pendant stars in $T$ or caterpillars in $T$, where these $S_i$ are all isomorphic. 
Let $T^*$ be the hyperforest we obtain by deleting the edges of $S_1,\dots, S_m$ from $T$ and further deleting the isolated vertices therein.
Then for each $S_i$, $V(S_i)\cap T^*$ consists of one vertex if $S_i$ is a pendant star and two vertices in different components of $T^*$ if it is a caterpillar.
We first wish to find an embedding $\varphi: T^*\rightarrow G$, and then embed all of $S_i$ with the root $V(S_i)\cap V(T^*)$ already fixed. 
Assuming that we have obtained the embedding $\varphi$, let $R_1=\bigcup_{i\in [m]}\varphi(V(T^*)\cap V(S_i))$ and $U=V\setminus \varphi(T^*)$.

An issue with this approach is that the embedding $\varphi$ may not behave nicely. For all we know, a vertex $v$ in $R_1$ or a vertex $u$ in $U$ may be isolated vertices in $G[R_1\cup U]$. In order to overcome this issue, we first partition $V$ into random sets $V_1,V_2,V_3$ before finding the embedding $\varphi$. As these sets are randomly chosen, we can ensure that every vertex $v$ in $V_i$ has many edges in crossing $(\{v\},V_j,\dots, V_j)$ for any $i,j\in [3]$.
Hence, once we find an embedding $\varphi$ of $T^*$ into $G[V_1]$, the remaining vertices in $V_1\setminus\varphi(T^*)$ have high degrees `towards' $V_2$ and $V_3$.
Still, some vertices $u\in V_i$ for $i\in \{2,3\}$ may not belong to any edges that intersect $R_1$. However, we can prove that these vertices are only a fraction of $V_2\cup V_3$ and belong to many edges in $G[\{u\}\cup V_{5-i}]$.
As our $S_i$ (which is either a pendant star or a caterpillar) contains a path of length at least two from `the root' $V(S_i)\cap V(T^*)$ to a leaf, we can embed such a path through $V_{5-i}$ connecting vertices in $R_1$ to such a vertex $u$. This will guarantee that we can first embed some of $S_i$ to occupy all such vertices $u$, overcoming the possible degree issue.
For this plan, we need the following three steps.
\begin{proplist}
        \item Find an almost spanning hyperforest $T^*$ in $G[V_1]$. \label{A1}
    \item Find a spanning collection of pendant stars in $G[R_1\cup U]$, where each $v\in R$ is covered by the center of a pendant star. \label{A2}
    \item Find a spanning collection of caterplillars in $G[R_1 \cup U]$, where each correct pair $v,v'\in R$ are connected by the caterpillar. \label{A3} 
\end{proplist}

The first step \ref{A1} is encapsulated in \Cref{lem:almost_spanning}.
This lemma is proved by adapting the proof method from \cite{Im2022proof}. We first show that a bounded degree hypertree $T$ yields a sequence $T_0 \subseteq T_1 \subseteq \cdots \subseteq T_\ell = T$ of subtrees for some $\ell = O(1)$ such that $|V(T_0)| = o(n)$ and each $T_{i+1}$ is obtained from $T_i$ by adding a matching for all $i\in [\ell]\setminus \{s\}$ for some $s$. For such exceptional index $s$, the tree $T_{s+1}$ is obtained from $T_s$ by adding $o(n)$ many paths of length three between vertices of $T_s$.
This sequence yields a decomposition of $T$ into a small tree $T_0$ and matchings and paths of length three, and this decomposition is useful for our purpose.
Using the lack of $r$-partitie holes and the minimum degree condition, we can add an almost perfect matching to each $T_i$. By setting aside a random vertex set $R$  at the beginning, we can utilize the minimum degree condition and the set $R$ to complete this almost perfect matching and obtain $T_{i+1}$. The minimum degree condition and the vertex set $R$ also allow us to obtain many disjoint paths of length three, hence help us to extend $T_s$ to $T_{s+1}$ as well, yielding a copy of $T^*$ in $G[V_1]$.
This will be stated in Section~\ref{sec:almost_spanning}.

For \ref{A2}, we use the absorption method together with the template method introduced by Montgomery~\cite{Montgomery2019spanning} to achieve the goals. For \ref{A3}, we need extra care in addition to the proof method for \ref{A2}. We utilize the weak hypergraph regularity method in \Cref{sec:weak_hypergraph_regularity}, to find an almost spanning collection of caterpillars connecting appropriate pairs of root vertices, and
we complete this to a spanning collection of caterpillars in \Cref{sec:absorbing} using the lattice-based absorbing method introduced by Han~\cite{Han2017decision} to conclude the proof.

\section{Preliminary}\label{sec:prelim}
\subsection{Basics}

For $n,n'\in \mathbb{N}$, we write $[n]=\{1,\dots, n\}$ and $[n,n']= \{n,n+1,\dots, n'\}$.
For two sets $X$ and $Y$, we write $X \triangle Y$ to denote the symmetric difference $(X\setminus Y)\cup (Y\setminus X)$. 
For $a, b, c>0$, we write $a=b \pm c$ if $a \in [b-c, b+c]$.
We use the symbol $\ll$ to denote the hiearchy of constants we choose. 
More precisely, if we claim that a statement holds with $0<a \ll b, c \ll d$, it means that there exist functions $f_1, f_2:(0, 1) \rightarrow (0, 1)$ and $g:(0, 1)^2 \rightarrow (0, 1)$ such that the statement holds with every $b \leq f_1(d), c \leq f_2(d)$ and $a \leq g(b, c)$. We do not explicitly compute these functions.
Throughout the paper, all hypergraphs we consider are $r$-uniform hypergraphs unless otherwise mentioned. Hence, we often omit its uniformity and just state that it is a hypergraph or a hypertree.

Consider an $r$-graph $G$. For a vertex $v \in V(G)$, we write $N_G(v)$ to denote the set $N_G(v) = \{e \setminus \{v\} \mid v \in e,  e \in E(G)\}$ of neighborhood of $v$. 
Similarly, we define $N_G(u,v) = \{e\setminus\{u,v\} \mid \{u,v\}\subseteq e, e\in E(G)\}$ as the common neighborhood of $u$ and $v$.
For a tuple $\mathcal{X}=(X_1,\dots, X_r)$ of (not necessarily disjoint) subsets of $V(G)$, we say that a set $\{x_1,\dots, x_r\}$ of $r$ distinct vertices is \emph{crossing} $\mathcal{X}$ if $x_i\in X_i$ for all $i\in [r]$. 
We write $e_G(\mathcal{X})= e_G(X_1,\dots, X_r)$ to count the number of tuples $(x_1,\dots, x_r) \in X_1\times \dots \times X_r$ with $\{x_1,\dots, x_r\}\in E(G)$. Note that this counts tuples, not edges, so an edge can be counted multiple times.
For a vertex $v$ and a tuple $\mathcal{Y}= (Y_1,\dots, Y_{r-1})$ of (not necessarily disjoint) subsets of $V(G)$, we write $d_G(v,\mathcal{Y}) = e_G(\{v\},Y_1,\dots, Y_{r-1})$.
We write $G[\mathcal{X}]=G[X_1, \ldots, X_r]$ for the $r$-partite subgraph of $G$ on the vertex set $X_1\cup \dots \cup X_r$ consisting of all edges crossing $\mathcal{X}$.


For a given tuple $\mathcal{V}=(V_1,\dots, V_s)$ of sets and a set $I\subseteq [s]$, we write $\mathcal{V}_I = (V_i: i\in I)$. For a hypergraph $H$ with $V(H)= [t]$, we say that $G$ \emph{admits a partition $(\mathcal{V},H)$} for some $\mathcal{V}=(V_1,\dots, V_t)$ if $G[\mathcal{V}_{e'}]$ is empty for all $r$-sets $e'\notin E(H)$ where $\mathcal{V}_{e'}= (V_i: i\in e')$.

The following lemma will be useful in many places\footnote{seems only used in the proof of Lemma 3.2} where we need to obtain a matching utilizing the $r$-partite hole condition.

\begin{lemma}\label{lem: find matching}
Suppose that $G$ is an $r$-graph and $\mathcal{V}=(V_0,\dots, V_{r-1})$ and $\mathcal{U}=(U_1,\dots, U_{r-1})$ are tuples of vertex sets where $V_0,\dots, V_{r-1}$ are pairwise disjoint satisfying $|U_i| = m$ for all $i\in [r-1]$.
Suppose that $\bigcup_{i\in [r-1]} U_i$ is disjoint from $\bigcup_{0\leq i\leq r-1} V_i$ and the following holds for any positive integer $m^*\leq m$.
\begin{proplist}
    \item $|V_0| \leq \min_{i\in [r-1]} |V_i|$.
    \item $\alpha^*_{\mathcal{V}}(G) \leq m^*$.
\end{proplist}
Then $G$ contains a matching $M'$ in $G[\mathcal{V}]$ covering at least $|V_0|-m^*$ vertices of $V_0$.
Furthermore, suppose the following holds in addition
\begin{enumerate}[label = {\bfseries \Alph{propcounter}\arabic{enumi}}] \setcounter{enumi}{2}
    \item $d_G(v,\mathcal{U}) \geq  (r-1)^2 m^{r-2}m^*$ for all $v\in V_0$.
\end{enumerate}
Then there exists an additional matching $M''$ such that $M= M'\cup M''$ is a matching that covers all the vertices of $V_0$ where each edge of $M''$ crosses $(V_0,U_1,\dots, U_{r-1})$ and $|M''|\leq m^*$.
\end{lemma}
\begin{proof}
    Take a maximum matching $M'$ in $G[V_0,\dots, V_{r-1}]$. As $G[ V_0\setminus V(M'),\dots, V_{r-1}\setminus V(M')]$ contains no edges, the assumption
    $\alpha^*_{\mathcal{V}}(G) \leq m^*$ implies  $|M'|\geq \min_{i\in [r-1]\cup\{0\}} |V_i| - m^* = |V_0|-m^*$.
    For the remaining vertices $v$ in $V_0\setminus V(M')$, we greedily choose disjoint edges $e_1,\dots,e_{|V_0|-|M'|} $ that cross $(\{v\},U_1,\dots, U_{r-1})$. Once we choose $i< m^*$ such edges, for each vertex $v\in V_0\setminus(V(M')\cup e_1\cup \dots\cup e_i)$,  
    there are still at least
    $$\frac{1}{(r-1)!}\left( d_G(v,\mathcal{U}) - (r-1)^2 i m^{r-2}\right) \geq \frac{1}{(r-1)!}\left((r-1)^2 m^* m^{r-2} - (r-1)^2 i m^{r-2}\right) > 0$$ available edges disjoint from $e_1,\dots, e_i$ and crossing $(\{v\},U_1,\dots, U_{r-1})$. 
    Here, recall that $d_G(v,\mathcal{U})$ counts tuples instead of edges, so an edge can be counted at most $(r-1)!$ times. Also, we subtract $(r-1)^2i m^{r-2}$ from $d_G(v,\mathcal{U})$ in the first term because there are $(r-1)i$ vertices $u$ in $e_1\cup\dots\cup e_{i}$ and such a vertex $u$ may belong to $U_j$ for at most $(r-1)$ choices of $j$ (recall that $U_1,\dots, U_{r-1}$ may not be disjoint). 
    As $|V_0|-|M'|\leq m^*$, we can choose $e_1,\dots, e_{|V_0|-|M'|}$. This yields a desired matching $M''=\{e_1,\dots, e_{|V_0|-|M'|}\}$.
\end{proof}

By repeatedly applying this lemma, one can prove the existence of a linear path as follows.

\begin{lemma}\label{lem: find path}
Suppose that $G$ is an $r$-graph and $\mathcal{V}=(V_1,\dots, V_{(r-1)t+1})$ is a tuple of pairwise disjoint sets 
with $|V_i|> ((r-1)t+1)\alpha_{\mathcal{V}}^{*}(G)$ for each $i\in [(r-1)t+1]$.
Then $G$ contains a linear path $P$ connecting a vertex in $V_1$ to a vertex in $V_{(r-1)t+1}$ whose vertex set crosses $\mathcal{V}$.
\end{lemma}
\begin{proof}
   Let $m$ be the minimum $|V_i|$ over all $i\in [(r-1)t+1]$.
   For each $i\in [(r-1)t+1]$, we choose a subset $V'_i\subseteq V_i$ of size exactly $m$. 
   For each $i\in [t]$, 
   apply \Cref{lem: find matching} for each $G[\mathcal{V}'_{e_i}]$
   where $e_i=\{1+(r-1)(i-1),\dots, 1+(r-1)i\}$. This yields a matching of size $m-\alpha_{\mathcal{V}}^{*}(G)$ in each $G[\mathcal{V}'_{e_i}]$.
   As $m > ((r-1)t+1)\alpha_{\mathcal{V}}^{*}(G)$, concatenating the edges in these matchings, we obtain at least one desired $\mathcal{V}$-transversal path.   
\end{proof}

Finally, the following lemma is useful for finding a set of large degree vertices.
\begin{lemma}\label{lem: find a set}
Suppoes $t>0$.
    Suppose that $G$ is an $r$-partitite $r$-garph with an $r$-partition $\mathcal{V}=(V_1,\dots, V_r)$ such that $|V_j|\leq n$ for all $j\in [r]$. 
    For $i\in [r]$ and a subset $U\subseteq V(G)$ satisfying $U\cap V_i=\emptyset$, let $W = \{ w\in V_i: d_G(U\cup \{w\})\geq t\}$. Then $|W|\geq \frac{1}{n^{r-|U|-1}-t+1}(d_G(U)- (t-1)n)$
 \end{lemma}
\begin{proof}
For each $w$, $w\cup U$ belongs to at most $n^{r-|U|-1}$ edges. 
Counting all edges of $G$ containing $U$, we have 
$d_G(U) \leq n^{r-|U|-1} |W| + (t-1)(n-|W|)$. Thus we have $|W|\geq \frac{1}{n^{r-|U|-1}-t+1}(d_G(U)-(t-1)n)$.
\end{proof}

\subsection{Concentration inequalities}\label{sec:prelim_concentration}

We also collect some useful probabilistic tools regarding concentration of random variables and their consequences on degrees on certain random subgraphs. 
We first need the following Chernoff bound (see, for example, \cite{Alonprobabilistic}). 
\begin{lemma}[The Chernoff Bound]\label{lem:chernoff}
    Let $X_1, X_2, \cdots, X_n$ be i.i.d. Bernoulli random variables and let $X= \sum_{i=1}^n X_i$. Then for $\varepsilon \in (0, 1)$, $$\mathbb{P}(|X-\mathbb{E}[X]| \geq \varepsilon \mathbb{E}[X] ) \leq 2\exp\left(-\frac{\varepsilon^2 \mathbb{E}[X]}{3}\right).$$     
\end{lemma}
A hypergeometric version of the above lemma will also be used frequently.
\begin{lemma}\label{lem:chernoff_hypergeometric}
    Let $0<m, r \leq n$. Let $S \subseteq [n]$ be a fixed subset of size $m$ and $X \subseteq [n]$ be a random subset of $[n]$ of size $r$ chosen uniformly at random. Then for every $t \geq 0$, we have 
    $$\mathbb{P}\left(\left||S \cap X| - \frac{mr}{n}\right| \ge t\right) \leq 2\exp\left(-\frac{t^2}{8r}\right)$$
\end{lemma}
We also need the following inequality regarding the degree of a vertex in a random partition of $r$-graphs. We provide proof of this lemma in the appendix.

\begin{lemma}\label{lem:general_concent}
Suppose $0<1/n \ll 1/t, 1/r, p_0 \leq 1$.
Let $G$ be an $n$-vertex $r$-graph with $\delta_1(G)\geq \varepsilon n^{r-1}$.
Let $p_1,\dots, p_t$ be positive numbers with $p_i\geq p_0$ for all $i\in [t]$ and $\sum_{i\in [t]} p_i=1$.
    Let $\mathcal{X}=(X_1, \ldots, X_t)$ be a random partition of $V(G)$ into $t$ parts, each of size $p_in$, chosen uniformly at random.
Then with probability $1-o(1)$, we have 
$$d_G(v,Y_1,\dots, Y_{r-1}) \geq \frac{1}{2} \varepsilon \prod_{i\in [r-1]} |Y_i|$$
for every $v\in V(G)$ and all choices of $Y_1,\dots, Y_{r-1}\in \{X_1,\dots, X_t\}$ with possible repetitions.
\end{lemma}

\subsection{Hypertrees}\label{sec:prelim_hypertree}
We also need to collect definitions and useful facts about hypertrees.
A hypergraph $H$ is \emph{linear} if $|e\cap f|\leq 1$ for all $e,f\in E(H)$.
A \emph{linear ($r$-)path $P_\ell^{(r)}$ of length $\ell$} is a hypergraph consisting of $(r-1)\ell+1$ vertices $v_0,\dots, v_{(r-1)\ell}$ and edges $\{v_{(r-1)(i-1)}, v_{(r-1)(i-1)+1}\dots, v_{(r-1)i}\}$ for each $i\in [\ell]$.
Each of the two sets $\{v_0,\dots, v_{r-2}\}$ and $\{v_{(r-1)(\ell-1)+1},\dots, v_{(r-1)\ell}\}$ of size $r-1$ is an \emph{end set} of the path. 
A \emph{$u$--$v$ path} $P$ (or a path $P$ between $u$ and $v$) is a path $P$ paired with the specified end vertices $u$ and~$v$, each of which is chosen from each end set, respectively.
The vertices other than $u$ and $v$ in a $u$--$v$ path $P$ are called the \emph{internal} vertices of $P$.

A linear $r$-graph $T$ is a \emph{($r$-)hypertree} if there is a unique linear path between any pair of distinct vertices. A \emph{($r$-)hyperforest} is a vertex-disjoint union of $r$-hypertrees. 
In a hypertree $T$, a \emph{bare path} $P$ is a path in $T$ of length $\ell \geq 2$ where no edges in $T\setminus E(P)$ are incident to any internal vertices of $P$.

A \emph{leaf-edge} of an $r$-uniform hypertree $T$ is an edge $e$ that contains at least $r-1$ vertices of degree one. 
A \emph{leaf} of a hypertree $T$ is a vertex $v\in V(T)$ of degree one in a leaf-edge of $T$ .
A \emph{star} of size $d$ with \emph{center} $v$ is an $r$-graph with $d$ edges $e_1,\dots, e_d$ with $e_i\cap e_j = \{v\}$ for all $i\neq j \in [d]$.
A \emph{matching}~$M$ is a collection of pairwise disjoint edges. 
If a matching $M$ consists of leaf-edges of a hypertree~$T$, the set of leaves of $T$ in $M$ is called the \emph{leaf set} of $M$. 
We simply say that a vertex subset $X$ of $T$ is a  \emph{matching leaf set} of $T$ if $X$ is the leaf set of a matching $M$ in $T$. 

For a hypertree $T$, let $T'$ be a hypertree obtained by deleting all the leaf-edges of $T$.
Let $e$ be a leaf-edge of $T'$ and $v \in e$ be a unique vertex of degree at least $2$ in $T'$ (if it does not exists, choose any).
A subgraph $S$ of $T$ consists with $e$ together with leaf-edges of $T$ incident to $e \setminus \{v\}$ is called a \emph{pendant star} of $T$.
The edge $e$ is called the \emph{center} and the vertex $v$ in $e$ is called the \emph{root} of the pendant star $S$.
A subgraph $L$ of $T$ that consists of a bare path $P$ of $T'$ of length $\ell$ and leaf-edges of $T$ incident to its internal vertices is called a \emph{caterpillar of length $\ell$}. 
The path $P$ is called the \emph{central path} of the caterpillar $L$.

It is well known that a ($2$-uniform) tree $T$ on $n$ vertices contains either many leaves or many bare paths (see, e.g., \cite{Krivelevich2010embedding}).
By using this, Han, Hu, Ping, Wang and Yang proved every tree has many pandent stars or caterpillars.
\begin{lemma}[\cite{Han2024Spanning}]\label{lem:pendant_and_caterpillars_for_graph}
    For an $n$-vertex tree $T$ with $\Delta(T) \leq \Delta$, either it contains at least $\frac{n}{4k\Delta}$ pendant stars or a collection of at least $\frac{n}{4k\Delta}$ vertex-disjoint caterpillars each of length $k$.
\end{lemma}
This result can be extended to hypertrees as follows.

\begin{lemma}\label{lem:pendant_and_caterpillars}
Let $r \geq 2$ and $t \geq 6$ and let $T$ be a $r$-uniform hypertree on $n$ vertices and $\Delta(T) \leq \Delta$ having at least one non-leaf-edge.
    Then $T$ contains at least $\frac{n}{12r^2(t+1)\Delta}$ vertex-disjoint non-trivial pendant stars or $\frac{n}{4r^2(t+1)\Delta}$ vertex-disjoint caterpillars each of length $t$.
\end{lemma}

It is easy to check that the following restatement of Lemma~A.2 in \cite{Im2022proof} follows from the proof of Lemma~A.2 in \cite{Im2022proof}.

\begin{lemma}[\cite{Im2022proof}, Lemma~A.2]\label{lem: leaf/bare paths}
Let $r,m\geq 2$ and let $T$ be a hypertree with at most $\ell$ leaf-edges. Then there exist pairwise edge-disjoint bare paths $P_1,\dots, P_s$ of length $m$ such that $e(T-P_1 -\dots -P_s) \leq 6(m+1)\ell + \frac{2e(T)}{m+2}$. In particular, if $\ell\leq \frac{e(T)}{12(m+1)}$ and $m\geq 6$, then $s\geq \frac{e(T)}{4m}$.
\end{lemma}

\begin{proof}[Proof of Lemma~\ref{lem:pendant_and_caterpillars}]
Note that $T$ contains $\frac{n-1}{r-1}$ edges.
Delete all leaf-edges of $T$ to obtain a hypertree $T'$. As $T$ has at least one non-leaf-edge, $T'$ contains at least one edge. Moreover, since every non-leaf-edge is adjacent to at most $r(\Delta-1)$ leaf-edges, we have $e(T')\geq \frac{e(T)}{r(\Delta-1)} \geq \frac{n}{r^2\Delta}.$

If $T'$ contains at least $\frac{n}{ 12r^2(t+1)\Delta}$ leaf-edges, then they correspond to that many  pendant stars. Moreover, these pendant stars are non-trivial pendant stars as their centers are not leaf-edges of $T$. Otherwise, Lemma~\ref{lem: leaf/bare paths} implies that it contains at least $\frac{e(T')}{4t} \geq \frac{n}{4r^2(t+1)\Delta}$ bare paths of length $t$. Again, each such bare path corresponds to a caterpillar in $T$.
\end{proof}

Later we will utilize the following lemma regarding a decomposition of hypertrees. Its graph version was originally proved by Montgomery, Pokrovskiy and Sudakov in \cite{Montgomery2020embedding} and its hypergraph version was considered by Im, Kim, Lee and Methuku~\cite{Im2022proof}.
\begin{lemma}[\cite{Im2022proof}]\label{lem:tree_split}
    Let $n \geq 2$ be an integer and let $0<\mu<1$.
    For any hypertree $T$ with at most $n$ edges, there exist integers $\ell \leq 10^5 r \Delta(T) \mu^{-2}$ and $s \in [\ell]$ and a sequence of subgraphs $T_0 \subseteq T_1 \subseteq \cdots \subseteq  T_\ell = T$ such that the following holds:
    \begin{proplist}
               \item\label{it:T0} $T_0$ has at most $\mu n$ edges.
       \item\label{it:matching} For $i \in [\ell-1]\setminus\{s\}$, 
       $T_{i+1}$ is obtained by adding a matching to $T_i$ such that $V(T_{i+1})\setminus V(T_i)$ is a matching leaf set of $T_{i+1}$.
       \item\label{it:path} $T_{s+1}$ is obtained by adding at most $\mu n$ vertex-disjoint bare paths of length $3$ to $T_s$ 
       such that every bare path we add is a $u$--$v$ path $P$ where $u,v\in V(T_s)$, and $V(P) \setminus \{u,v\}$ is disjoint from $V(T_s)$. 
    \end{proplist}
\end{lemma}

\subsection{Weak hypergraph regularity}
\label{sec:weak_hypergraph_regularity}
In order to achieve \ref{A3}, we need to obtain many disjoint copies of cycles in a desired graph $G$. We plan to achieve this using the condition on $\alpha^*(G)$, which yields many paths. However, closing these paths into cycles requires additional dense and robust structures, which can be derived from the minimum degree condition using the tool of weak hypergraph regularity. In this section, we introduce weak hypergraph regularity.

\begin{defn}
For an $r$-graph $G$, a tuple $(X_1, \ldots, X_r)$ of disjoint vertex sets is $(\tau, d)$-regular if for every $Y_i \subseteq X_i$ for each $i \in [r]$ with $|Y_i| \geq \tau |X_i|$, we have 
$$|d - \frac{e_G(Y_1, \ldots, Y_r)}{ \prod_{i\in [r]} |Y_i| }| \leq \tau.$$
We say that it is $\tau$-regular if it is $(\tau,d)$-regular for some $d$ and we say it is $(\tau, d+)$-regular if it is $(\tau, d')$-regular for some $d' \geq d$.
\end{defn}

The following weak hypergraph regularity lemma is a generalization of the Szemer\'edi regularity lemma for graphs and this follows the lines of the original proof of Szemer\'edi in \cite{Szemeredi1978regular} (see, e.g. \cite{Chung1991regularity, Frankl1992uniformity, Steger1990kleitman}).

\begin{lemma}[Weak hypergraph regularity lemma]\label{thm:weak_hypergraph_regularity}
Let $0< 1/n \ll 1/T \ll 1/k, 1/p, 1/r, \tau \leq 1$.
    Let $G$ be an $n$-vertex $r$-graph and $\mathcal{V}=(V_1, \ldots, V_k)$ be an equitable partition of $V(G)$. Then there exists a partition $\mathcal{W}=(W_0, W_1, \ldots, W_t)$ for some $p \leq t \leq T$ of $V(G)$ such that the following hold.
    \begin{proplist}
          \item $|W_0| \leq \tau n$. \label{reg1}
        \item For every $i \in [t]$, $W_i$ is a subset of some $U_j$. \label{reg2}
        \item For every $i, j \in [t]$, $|W_i|=|W_{j}|$. \label{reg3}
        \item For all but at most $\tau {t \choose r}$ $r$-sets $I$ in $\binom{[t]}{r}$, the tuple $\mathcal{W}_I$ is $\tau$-regular. \label{reg4}
    \end{proplist}
\end{lemma}
A partition satisfying \ref{reg1}--\ref{reg4} is called a \emph{$\tau$-regular partition} of $G$ refining $\mathcal{V}$. If $\mathcal{V}=(V(G))$, then we simply say that it is a $\tau$-regular partition of $G$. We also define the clustered hypergraph, which is an analogue of the reduced graph.
\begin{defn}
For a given $\tau$-regular partition $\mathcal{W}=(W_0,W_1,\dots, W_t)$ of $G$, \emph{the clustered graph} $\mathcal{H}(\tau,d,\mathcal{W})$ is an $r$-graph on the vertex set $[t]$ such that an $r$-set $I\in \binom{[t]}{r}$ is an edge if and only if the tuple $\mathcal{W}_I$ is $(\tau,d+)$-regular.
\end{defn}

\begin{lemma}\label{lem: reg matching}
Let $0<1/n \ll 1/T \ll \tau \ll \varepsilon, 1/r \leq 1$ and $ \tau n < n^* < \varepsilon n/(8r)^r$.
    Let $G$ be an $r$-partite $r$-graph with the $r$-partition $\mathcal{V}=(V_1 \cup V_{r+1}, V_2\cup V_{r+2},\dots, V_r\cup V_{2r})$ and $(1-\varepsilon)n\leq |V_i|\leq (1+\varepsilon)n$ for all $i\in [2r]$.
    If each of $G[V_1,\dots, V_r]$ and $G[V_{r+1},\dots, V_{2r}]$ contains at least $\varepsilon n^{r}$ edges, then we can find subsets $V_{i,j}\subseteq V_i$ for each $(i,j)\in [2r]\times [p]$ for some $p\leq T$ such that the following hold.
    \begin{proplist}
\item For each $j,j'\in [p]$, we have $|V_{1,j}|=\dots = |V_{2r,j}| \geq \tau n/T$ and 
$\big||V_{1,j}| - |V_{1,j'}|\big|\leq 1$. \label{reg matching 1}
\item For each $j\in [p]$, each of the tuples $(V_{1,j},\dots, V_{r,j})$ and $(V_{r+1,j},\dots, V_{2r,j})$ is $(\tau,\tfrac{\varepsilon}{10}+)$-regular. \label{reg matching 2}
\item $\sum_{j\in [p]} |V_{1,j}| = n^*$. \label{reg matching 3}
    \end{proplist}
\end{lemma}
\begin{proof}
By deleting edges if necessary, we assume that $G=G_1\cup G_2$ where $G_1=G[V_1,\dots, V_r]$ and  $G_2=G[V_{r+1},\dots, V_{2r}]$.
Note that $G_i$ has $(1\pm \varepsilon)rn$ vertices for each $i=1,2$.

We first apply \Cref{thm:weak_hypergraph_regularity} to $G$ to obtain a $\tau/3$-regular partition $\mathcal{W} = (W_0,W_1,\dots, W_t)$ with $ \varepsilon^{-1} <  t < T$ such that $\mathcal{W}$ refines $(V_1,V_2,\dots, V_{2r})$.
Let $m=|W_1|=\dots=|W_t|$. Then $mt = (2\pm 3\varepsilon)rn$.

By reindexing if necessary, assume that $W_i \subseteq V(G_1)$ if $i\in [t_1]$ and $W_i\subseteq V(G_2)$ if $i\in [t]\setminus[t_1]$ for some $t_1$.
Let $t_2= t-t_1$.
Consider the clustered graph $H= \mathcal{H}(\tau/3,d,\mathcal{W})$ for $d=\varepsilon/(10r)^r$ and let $H_1 = H[[t_1]]$ and $H_2= H[[t]\setminus [t_1]]$.
Then each $H_i$ is a clustered graph of $G_i$ for $i=1,2$. Furthermore, as each $G_i$ has $(1\pm \varepsilon)rn$ vertices, we have $t/3\leq t_1,t_2\leq 2t/3$.

By \ref{reg1}--\ref{reg4}, for $i\in [2]$ and $I\notin E(H_i)$, $\mathcal{W}_I$ is either $(\tau/3,\delta)$-regular with $\delta<d$ or not  $\tau/3$-regular at all.
As the latter case happens for at most $\frac{1}{3}\tau\binom{t_i}{r}$ tuples, we have
$$ e(H_i) m^{r} + \left|\binom{t_i}{r}-e(H)\right| d m^r + \frac{1}{3} \tau \binom{t}{r} m^r +|W_0|(2rn)^{r-1}\geq  \varepsilon n^{r}.$$
As $ m = (2\pm 3\varepsilon)rn/t$ and $\tau\ll \varepsilon$, this implies that 
$$e(H_i) \geq\frac{1}{m^r} \left(\varepsilon n^{r} -  \frac{d (2rn)^{r}}{r!} - \frac{\tau (2rn)^{r}}{r!} -\tau (2rn)^{r} \right)\geq \frac{1}{(3r)^r}\varepsilon t^{r}.$$
This implies that each $H_i$ contains a matching of size at least $\frac{1}{(4r)^r}\varepsilon t > 2n^*/m$.

For each $i\in [2]$, 
take a matching in $H_i$ of size exactly $p=\lceil 2n^*/m\rceil$. Then this corresponds to a set of $2p$ disjoint $\tau/3$-regular tuples, denoted as $(W_{(i-1)r+1,j},\dots, W_{(i-1)r+r,j})$  for all $j\in [p]$ (renamed to ease the notation), and $\sum_{j\in [p]} |W_{(i-1)r+1,j}| = p m = 2n^*\pm p$.

From the definition of $H_i$, we may further assume that the sets $W_{(i-1)r+k,j}$ lie in $V_{(i-1)r+k}$. 
As $2n^*\pm p$ has size between $n^*$ and $3n^*$, we can choose for all $i\in [2r]$ and $j\in [p]$ arbitrary subsets $V_{i,j}\subseteq W_{i,j}$ with $\frac{1}{3} |W_{i,j}| \leq |V_{i,j}| \leq |W_{i,j}| $ such that \ref{reg matching 1},\ref{reg matching 3} hold.
Finally, since $V_{i,j}$ is a subset of size at least $\frac{1}{3}|W_{i,j}|$, it is straightforward to check from the definition of $\tau$-regularity that \ref{reg matching 2} holds. 
\end{proof}

\subsection{Absorption}
\label{sec:absorbing}
As we discussed in \Cref{sec:outline}, we will use the absorption method to complete an almost-spanning structure into a spanning structure. For this, we need the following definitions.

\begin{defn}
Let $F$ be an $r$-graph on the vertex set $[t]$ and consider an $r$-graph $G$ admitting a partition $(\mathcal{V},F)$ where $\mathcal{V}=(V_1,\dots, V_t)$ and $V(G)=\bigcup_{i\in [t]} V_i$.     A set $S \subseteq V(G)$ is called \emph{$\mathcal{V}$-balanced} if $|S \cap V_i|$ is equal for every $i \in [t]$. If $|S \cap V_i|=1$ for all $i \in [t]$, then $S$ is called a \emph{$\mathcal{V}$-transversal set}. 

    For a real number $\gamma>0$ and $W \subseteq V(G)$, a set $S \subseteq V(G)$ is called a \emph{$(\gamma, F)$-absorbing set} (or simply \emph{$\gamma$-absorbing set}) for $W$ if for every $\mathcal{V}$-balanced set $S' \subseteq W$ with $|S'| \leq t \gamma n$, $G[S \cup S']$ have a $\mathcal{V}$-transversal $F$-factor.
    If $W = V(G)$, we omit $W$ and simply call $S$ a $\gamma$-absorbing set. 
    For a $\mathcal{V}$-balanced set $S \subseteq V(G)$ of size $t$, a set $S' \subseteq V(G) \setminus S$ is called an \emph{$(F, k)$-absorber} (of $S$) if $|S'| \leq kt$ and both of $G[S']$ and $G[S \cup S']$ have a $\mathcal{V}$-transversal $F$-factor. 
    For a vertex $v \in V(G)$, a collection $\mathcal{F}$ of $\mathcal{V}$-transversal copies of $F$ is called a \emph{$v$-fan} if the intersection of any two elements of $\mathcal{F}$ is exactly $\{v\}$.
\end{defn}

With this definition, we later need to find an absorbing set for appropriate $F$, where $F$ is either an edge or a cycle. For this pupose, the following `bipartite template lemma' by Montgomery~\cite{Montgomery2019spanning} is useful.
\begin{lemma}[\cite{Montgomery2019spanning}]\label{lem:bipartite_template}
    Let $0< 1/m \ll \beta< 1$.
    Then there exists a bipartite graph $H$ with bipartition $(X\cup Y, Z)$
    with the following properties.
    \begin{enumerate}
        \item $|X|=(1+\beta)m$, $|Y| = 2m$, $|Z| = 3m$ such that for every $X' \subseteq X$ with $|X'|=m$, $H[X' \cup Y, Z]$ has a perfect matching
        \item $\Delta(H) \leq 40$.
    \end{enumerate}
\end{lemma}

The following lemma provides criteria for finding a desired absorbing set. 

\begin{lemma}\label{lem:absorbing_using_template}
Let $0<1/n \ll \gamma \ll \eta, 1/r, 1/t, 1/k \leq 1$ and let $F$ be an  $r$-graph on the vertex set $[t]$.
    Let $G$ be an $r$-graph that admits a partition $(\mathcal{V},F)$ where $\mathcal{V}=(V_1, \ldots, V_t)$.
    Suppose that every $\mathcal{V}$-balanced set $S \subseteq V(G)$ of size $t$ has at least $\eta n$ pairwise disjoint $(F, k)$-absorbers of $S$.
    Then $G$ has a $\gamma$-absorbing set of size at most $\eta n$.
\end{lemma}
\begin{proof}[Proof of \Cref{lem:absorbing_using_template}]
Choose another constant $\delta$ such that
$$0<1/n \ll \gamma \ll \delta \ll \eta, 1/r, 1/t, 1/k\leq 1.$$

We first claim that for every vertex $v \in V(G)$, there is $v$-fan $\mathcal{F}_v$ of size at least $\eta n/t$, where each copy of $F \in \mathcal{F}_v$ is crossing $\mathcal{V}$. Indeed, if a maximum $v$-fan $\mathcal{F}_v$ has size smaller than $\eta n/t$, one may choose a balanced set $S$ of size $t$ containing $v$ disjoint from $\mathcal{F}_v\setminus\{v\}$. By the hypothesis, there are $\eta n$ pairwise disjoint $(F,k)$-absorbers $S'_1,\dots, S'_{\eta n}$, at least one $S'_{i}$ of them is disjoint from $\mathcal{F}_v$. As $S\cup S'_{i}$ contains a $\mathcal{V}$-transversal $F$-factor, we may choose a copy of $F$ containing $v$. Adding this to $\mathcal{F}_v$ contradicts the maximality of $\mathcal{F}_v$, hence there is $v$-fan $\mathcal{F}_v$ of size at least $\eta n$ for every $v\in V(G)$.

Let $m:=\delta n$.
For each $i\in [t]$, we choose a vertex set $X_i\subseteq V_i$ of size $(1+\delta)m$ such that the following holds where $X=\bigcup_{i\in [t]} X_i$.
\begin{proplist}
    \item For every $v\in V(G)$, $|\mathcal{F}'_v| \geq t^2 \delta m$. \label{F'v}
\end{proplist}
 where $\mathcal{F}'_v = \mathcal{F}_v[X\cup\{v\}]$ is the subgraph of $\mathcal{F}_v$ whose edges all lie in $X\cup \{v\}$. 
Indeed, if we choose a vertex set $X_i\subseteq V_i$ of size $(1+\delta)m$ uniformly at random, then Chernoff inequality (\Cref{lem:chernoff}) yields that \ref{F'v} holds with probability $1-o(1/n)$ for each $v \in V(G)$. Hence a union bound yields that a desired choice of $X_i$'s exist.
We fix such a choice of $X_1,\dots, X_t$.
Now we choose pairwise disjoint sets $Y_i, Z_i$ in $V(G)\setminus X$ in an arbitrary way that satisfy the following.
$$ Y_i \subseteq V_i\setminus X_i \text{ with } |Y_i|=2m \text{ and } Z_i\subseteq \bigcup_{j\in [t]\setminus \{i\} } (V_j\setminus (X_j\cup Y_j)) \text{ with }|Z_i\cap V_j|=3m \text{ for each } j\in [t]\setminus \{i\}. $$
As each $Z_i$ contains exactly $3m$ vertices from each $V_j$ with $j\neq i$, we can partition $Z_i$ into $3m$ sets $I_{i,1},\dots, I_{i,3m}$ of size $t-1$ each crosses $(V_1,\dots, V_{i-1}, V_{i+1},\dots, V_t)$. Let $\mathcal{Z}_i = \{ I_{i,1},\dots, I_{i,3m}\}$.

By applying \Cref{lem:bipartite_template} with $m$ and $\delta$ playing the roles of $n$ and $\beta$, respectively, we obtain a biparite graph $B$ satisfying the following.
\begin{proplist}
    \item $B$ has bipartition $(X'\cup Y', Z')$ such that $|X'|=(1+\delta)m$, $|Y'|=2m$ and $|Z'|=3m$.
    \item $|E(B)|\leq 300 m$.
    \item For any subset $X''\subseteq X'$ with $|X''|=m$, $B[X''\cup Y', Z']$ contains a perfect matching. \label{eq: pm in template}
\end{proplist}
For each $i\in [t]$, we consider a bijection $\varphi_i$ from $X'\cup Y'\cup Z'$ to $X_i\cup Y_i\cup \mathcal{Z}_i$, where $\varphi_i(X')=X_i$, $\varphi_i(Y')=Y_i$ and $\varphi_i(Z')=\mathcal{Z}_i$.
For each edge $e=ab\in B$, the set $\varphi_i(e)=\varphi_i(a)\cup \varphi_i(b)$ is a $\mathcal{V}$-balanced set of size $t$, hence it has at least $\eta n$ disjoint $(F,k)$-absorber. We choose one of them to be $F_{i,e}$. Furthermore, as $t \cdot |F|\cdot |E(B)|\leq 1000 t^2 m < \eta n/t$, we can make such choices so that $V(F_{i,e})\cap V(F_{i',e'})  = \varphi_i(e)\cap \varphi_{i'}(e')$ for all $i\in [t]$ and $e,e'\in E(B)$. 
Let $A$ be the subgraph of $G$ with $A=\bigcup_{i\in [t]} \bigcup_{e\in E(B)} F_{i,e}$.
Then $A$ satisfies the following.
    \begin{claim}\label{claim:absorbing_using_template}
        For any choice of sets $X^*_i \subseteq X_i$ with $|X^*_i| = \delta m$ for each $i \in [t]$, the graph $G[A \setminus \bigcup_{i \in [t]} X^*_i]$ has a transversal $F$-factor.
    \end{claim}    
    \begin{proof}
    For each $i$, consider $\varphi_i^{-1}(X^*_i)\subseteq X'$.
    By \ref{eq: pm in template}, $B\setminus \varphi_{i}^{-1}(X^*_i)$ contains a perfect matching $M_i$ for each $i\in [t]$.
 As each $F_{i,e}$ is an $(F,k)$-absorber of $\varphi_i(e)$, the subgraph 
    $$L_i:=\bigcup_{e\in M_i} \left(F_{i,e}\cup \varphi_i(e)\right) \cup \bigcup_{e\in E(B)\setminus M_i} F_{i,e}$$
    forms a vertex-disjoint collection of transversal copies of $F$ covering 
    $(X_i\setminus X_{i}^*)\cup Y_i\cup Z_i \cup \bigcup_{e\in E(B)} F_{i,e}.$
    Thus, taking the union of $L_1,\dots, L_{t}$, we obtain a transversal $F$-factor in $G[A \setminus \bigcup_{i \in [t]} X^*_i]$. This proves the claim.
    \end{proof}

    We now complete the proof by showing that $A$ is a desired $\gamma$-absorbing set.
    Consider a $\mathcal{V}$-balanced set $U\subseteq V(G) \setminus A$ with $|U| = t m'$ with $m'\leq \gamma n$.
    As $m' \leq  \gamma n < \delta m/t$, we can take an arbitrary set $U'\subseteq X_1$ of size $\delta m - tm'$.
    
    By \ref{F'v}, for each $v\in U\cup U'$, we have a $v$-fan $\mathcal{F}'_v$ in $G[ \{v\}\cup X]$ of size at least $t^2 \delta m$. Hence, for each $v\in U\cup U'$, we may choose a $\mathcal{V}$-balanced copy $F_v$ in $\mathcal{F}'_v$ so that all of them are pairwise vertex-disjoint. Indeed, this is possible because $\gamma \ll \delta$ implies $t|U\cup U'| \leq  t^2\delta m = t^2\delta^2 n$. 
    Recall that if $v\in U\cap V_i$, then $F_v$ crosses $(X_1,\dots,X_{i-1},\{v\}, X_{i+1},\dots,X_t)$ and $v\notin X_i$,
    and if $v\in U'$, then $F_v$ crosses $(\{v\},X_2,\dots, X_t)$.
    Hence, as $U$ is a $\mathcal{V}$-balanced set, for each $i\in [t]$, $X^{*}_i= X_i\cap \bigcup_{v\in U\cup U'} F_v$ has size exactly $|U| + |U'| = \delta m$. Hence \Cref{claim:absorbing_using_template} implies that the graph $G[A \setminus \bigcup_{i \in [t]} X^*_i]$ has a prefect $\mathcal{V}$-transversal $F$-factor. This together with the graphs $F_v$ for each $v\in U\cup U'$ yields a transversal $F$-factor in $G[A\cup U]$, showing that $A$ is a desired $\gamma$-absorbing set.
\end{proof}

Using this lemma, the following corollary is immediate. 
\begin{cor}\label{cor:perfect_transversal_factor}
    Let $0<1/n \ll \delta \ll \eta, 1/r, 1/t, 1/k\leq 1$.
    Let $F$ be a $t$-vertex $r$-graph and an $r$-graph $G$ admits a partition $(\mathcal{V},F)$ where $\mathcal{V}=(V_1,\dots, V_t)$ with $|V_i|=n$ for each $i\in [t]$.
    Suppose that $G$ satisfies the following two properties.
    \begin{enumerate}
        \item For every transversal set $S \subseteq V(G)$, there are at least $\eta n$ pairwise disjoint $(F, k)$-absorbers of $S$.\label{cond:perfect_matching_cond_1}
        \item For every balanced set $U \subseteq V(G)$ with $|U| \leq t \eta n$, $G \setminus U$ has a transversal $F$-factor that covers all but at most $t \delta n$ vertices.\label{cond:perfect_matching_cond_2}
    \end{enumerate}
    Then $G$ has a transversal $F$-factor.
\end{cor}

Later in the proof, we plan to use \Cref{lem:absorbing_using_template} to obtain an absorber. For this, we need to show that every given $\mathcal{V}$-balanced set $S$ has many disjoint $(F,k)$-absorber, but it is sometimes a non-trivial task. The following concept is useful for proving the existence of many $(F,k)$-absorbers.
\begin{defn}
    Let $F$ be an $r$-graph on vertex set $[t]$ and $G$ be an $r$-graph admitting $(\mathcal{V},F)$-partition where $\mathcal{V}=(V_1,\dots, V_t)$.
    For $i\in [t]$ and a pair of vertices $u, v \in V_i$, we say that $u$ and $v$ are \emph{$(F, m, k)$-reachable} if for every set $W \subseteq V(G)\setminus \{u, v\}$ satisfying $|W \cap V_j| \leq m$ for each $j \in [t]$, there exists a set $S \subseteq V(G)$ of size at most $kt-1$ such that both $G[S \cup \{u\}]$ and $G[S \cup \{v\}]$ have a transversal $F$-factor.
    A vertex set $U \subseteq V_i$ for some $i \in [t]$ is called \emph{$(F, m, k)$-closed} if any pair of two vertices in $U$ are $(F, m, k)$-reachable.
\end{defn}
We introduce another useful definition as follows.

\begin{defn}
    We say that a vertex $v$ is \emph{weakly $(F,m,k)$-reachable} to a vertex set $U$ if for every set $W \subseteq U \setminus \{v\}$ satisfying $|W\cap V_J|\leq m$ for each $j\in [t]$, there exist a vertex $u\in U\setminus W$ and a set $S\subseteq V(G)$ of size at most $kt-1$ such that $G[S\cup \{u\}]$ and $G[S\cup \{v\}]$ has a transversal $F$-factor.
\end{defn}
Note that there are subtle difference here that the vertex $u$ is chosen after fixing $W$ in this definition of weak reachablility.
This notion of weak reachability is useful for us to prove the following `transitivity' for reachability.
\begin{prop}\label{prop: transitivity}
Let $F$ be an $r$-graph on vertex set $[t]$ and $G$ is an $r$-graph admitting $(\mathcal{V},F)$-partition where $\mathcal{V}=(V_1,\dots, V_t)$.
Let $v,w$ be two vertices in $V_i$ and $U$ be a subset of $V_i$ and assume that the following holds.
\begin{proplist}
    \item $v$ and $u$ are $(F,m,k_1)$-reachable for all $u\in U$.\label{eq: reachable}
    \item $w$ is weakly $(F,m,k_2)$-reachable to $U$.\label{eq: weakreachable}
\end{proplist}
 Then $v,w$ are $(F,m-k_2,k_1+k_2 )$-reachable.
\end{prop}
\begin{proof}
    Let $W$ be a set satisfying $|W\cap V_j|\leq m-k_2$ for all $j\in [t]$.
    Using \ref{eq: weakreachable}, we can find a vertex $u\in U$ and
    set $S_1\subseteq V(G)\setminus W$ such that both $G[S_1\cup \{u\}]$ and $G[S_1\cup \{w\}]$ contain transversal $F$-factor on $k_2t$ vertices.

    Let $W'=W\cup S_1 \cup \{w\}$. Using \ref{eq: reachable}, we can find a set $S_2$ disjoint from $W'$ such that both $G[S_2\cup \{u\}]$ and $G[S_2\cup \{v\}]$ contain transversal $F$-factors on $k_1t$ vertices.
    Then for $S=S_1\cup S_2\cup \{u\}$, which is disjoint from $W$, both $G[S\cup \{w\}]$ and $G[S\cup \{v\}]$ contain transversal $F$-factors on $(k_1+k_2)t$ vertices. This shows that $v,w$ are $(F,m-k_2,k_1+k_2)$-reachable.
\end{proof}

Note that $w$ is weakly $(F,m,k_2)$-reachable to $U$ only if $U$ has size at least $m+1$. Also if $U\cup \{v\}$ is $(F,m,k)$-closed and $|U|\geq m+1$, then $v$ is weakly $(F,m,k)$-reachable to $U$. Later we will prove that weak reachability is a useful way to extend a closed set into a bigger closed set (see Lemma~\ref{lem:merge_closed_and_nonclosed_sets}).
Moreover, this notion of closedness is useful for us to obtain absorbing sets as follows.

\begin{lemma}\label{lem:closed_to_absorbing}
    Let $0<1/n \ll \gamma \ll \eta, 1/r, 1/t, 1/k$ and $F$ be an $r$-graph on the vertex set $[t]$ and $G$ be an $r$-graph admitting partition $(\mathcal{V},F)$ for $\mathcal{V}=(V_1,\dots, V_t)$.
    Suppose that $G$ satisfies the following properties.
    \begin{proplist}
        \item For every $i \in [t]$, $V_i$ is $(F, 2kt\eta n, k)$-closed. \label{closed 1}
        \item For every set $W$ with $|W| \leq 2kt^2\eta n$, $V(G) \setminus W$ contains at least one transversal copy of $F$. \label{closed 2}
    \end{proplist}
    Then $G$ has a $\gamma$-absorbing set of size at most $\eta n$.
\end{lemma}
\begin{proof}
Choose a constant $\beta$ so that 
$$0<\gamma \ll \beta \ll \eta, 1/r, 1/t, 1/k \leq 1.$$
By \Cref{lem:absorbing_using_template}, it suffices to show that every $\mathcal{V}$-balanced set of size $t$ has at least $\eta n$ disjoint $(F, kt)$-absorbers.
    Let $S$ be a transversal set and we choose a maximal collection of vertex-disjoint $(F, kt)$-absorbers $A_1, \ldots, A_m$ of $S$.
    Suppose that $m < \eta n$ and let $W = \bigcup_{i \in [m]} A_i \cup S$.
    Then $|W| \leq kt^2 \eta n$ and thus by \ref{closed 2}, $V(G) \setminus W$ contains a transversal copy of $F$.
    Let $S'$ be the vertex set of such a copy of $F$.
    Let $v_i$ be the vertex in $S \cap V_i$ and $u_i$ be the vertex in $S' \cap V_i$ for each $i \in [t]$.
    Then $u_1$ and $v_1$ are $(F, 2kt\eta n, k)$-reachable so there exists a set $S_1 \subseteq V(G) \setminus (W \cup S')$ of size at most $kt-1$ such that $G[S_1 \cup \{u_1\}]$ and $G[S_1 \cup \{v_1\}]$ have a transversal $F$-factor.
    By applying the same process for each $i \in [t]$, we can find a set $S_1, \ldots, S_t$ of size at most $kt-1$ such that $G[S_i \cup \{u_i\}]$ and $G[S_i \cup \{v_i\}]$ have a transversal $F$-factor for each $i \in [t]$, they are pairwise disjoint and disjoint from $W \cup S'$. Indeed, this is possible as $W\cup S'\cup S_1\cup\dots \cup S_{i-1}$ has size at most $2kt^2\eta n$ for $i<t$.
Then both $G[S'\cup \bigcup_{i \in [t]} S_i ]$ and $G[S'\cup \bigcup_{i \in [t]} S_i \cup S]$ contain $\mathcal{V}$-balanced $F$-factors.
Thus, $S'\cup \bigcup_{i \in [t]} S_i$ form an $(F, kt)$-absorber of $S$ so it contradicts the maximality of $A_1, \ldots, A_m$.
\end{proof}

Note that \ref{closed 2} is guaranteed by the minimum degree condition when $\eta$ is small. Hence, we need to show closedness to obtain a desired absorber. For this purpose, we collect several lemmas useful for showing the closedness. The following lemma is a slight variation of Lemma 5.4 in \cite{Han2021nonlinear}. As this can be proved using a proof almost the same to the one in \cite{Han2021nonlinear} utilizing Lemma~6.3 of Han and Treglown~\cite{Han2019complex}, we omit the proof here.
\begin{lemma}\cite{Han2021nonlinear}\label{lem:partition_into_closed}
    Let $0<1/n \ll \beta \ll \beta' \ll \delta \ll 1/c, 1/r, 1/t$.
    Let $F$ be an $r$-graph on the vertex set $[t]$ and $G$ be an $r$-graph admitting a partition $(\mathcal{V},F)$ where $\mathcal{V}=(V_1,\dots, V_t)$. 
    Suppose that for any $c+1$ vertices in each $V_i$, there exist two vertices that are $(F,\beta'n,1)$-reachable. 
    Then for every $i \in [t]$, $V_i$ can be partitioned into $V_1^i, V_2^i, \ldots, V_\ell^i, U^i$ for some $\ell \leq c$ such that $V_1^i, \ldots, V_\ell^i$ are $(F, \beta n, 2^{c-1})$-closed, $|V_j^i| \geq \frac{1}{2c}\delta n$ for each $j \in [r]$, and $0 \leq |U^i| \leq \delta n$.
\end{lemma}
The next lemma provides a way to show the closedness of each $V_i$.

\begin{lemma}\label{lem:merge_closed_and_nonclosed_sets}
Let $m,k,t$ be integers such that $m>10kt$.
    Let $F$ be an $r$-graph on the vertex set $[t]$ and $G$ be an $r$-graph admitting a partition $(\mathcal{V},F)$ where $\mathcal{V}=(V_1,\dots, V_t)$. 
    Fix $i\in [t]$ and let $U_1 \subseteq V_i$ be an $(F, m, k)$-closed set of size at least $m+10kt$ and $U_2 \subseteq V_i$ be a set disjoint from $U_1$.
    Suppose that for any $W \subseteq V(G)$ with $|W \cap V_j| \leq m - kt$ for all $j \in [t]$ and for any $v\in U_2$, there exist two vertex-disjoint transversal copies $F_1, F_2$ of $F$ in $G \setminus W$ such that 
\begin{proplist}
       \item $F_1 \cap V_i \subseteq U_1$ and $F_2 \cap V_i =\{v\}$. \label{two closed 1}
       \item For every $j \in [t] \setminus \{i\}$, two vertices $F_1 \cap V_j$ and $F_2 \cap V_j$ are $(F, m, k)$-reachable.  \label{two closed 2}
\end{proplist}
    Then $U_1 \cup U_2$ is $(F, m-10kt, 6kt)$-closed.
\end{lemma} 
\begin{proof}
To show the $(F, m-10kt, 6kt)$-closednes of $U_1\cup U_2$, it is sufficient to show that $v$ is weakly $(F,m-2kt,2kt)$-reachable to $U_1$.
Indeed, if this is true, then Proposition~\ref{prop: transitivity} and the closedness of $U_1$ implies that $\{v\}\cup U_1$ is $(F,m-4kt,(2t+1)k)$-closed set. Thus for any vertices $v,v'\in U_2$, $v$ is weakly $(F,m-4kt,(2t+1)k)$-reachable to $U_1$ and $v'$ and $u$ are $(F,m-4kt,(2t+1)k)$-reachable for all $u\in U_1$, Proposition~\ref{prop: transitivity} again implies that $v,v'$ are $(F,m-10kt, 6kt)$-reachable.
Hence it is enough to show that $v$ is weakly $(F,m-2kt,2kt)$-reachable to $U_1$.

    Let $W \subseteq V(G)$ be a set with $|W \cap V_j| \leq m-2kt$ for all $j \in [t]$.
    Then there exist two vertex-disjoint transversal copies $F_1, F_2$ of $F$ in $G \setminus W$ that satisfies \ref{two closed 1}--\ref{two closed 2}.
    For each $j\in [t]$, we let $u_j$ and $v_j$ be the unique vertex in  $V(F_1)\cap V_j$ and $V(F_2)\cap V_j$, respectively. In particular, $v_i=v$ and $u:=u_i$ is some vertex in $U_1$.
    Hence we aim to find a set $S$ of size at most $kt(t-1)-1$ such that both $G[S\cup \{u\}]$ and $G[S\cup \{v\}]$ contains transversal $F$-factors.

Assume that we have pairwise disjoint sets $S_1,\dots, S_{j-1}$ of size $kt-1$ in $V(G)\setminus (W\cup \{u,v\})$ for some $j\in [t]\setminus\{i\}$, such that both $G[S_{j'}\cup \{u_{j'}\}]$ and $G[S_{j'}\cup v_{j'}]$ contains a $\mathcal{V}$-transversal $F$-factor. Hence each $S_{j'}$ intersect $V_{j''}$ at $k$ vertices if $j''\neq j'$ and at $k-1$ vertices if $j''= j'$.
Letting $W' = W\cup \{u,v\}\cup\bigcup_{j'\in [j-1]} S_{j'}$, we have $|W'\cup \{u,v\} \cap V_{j''}|\leq m - 2kt +2 + \sum_{j'\in [j']\setminus \{i\} } k \leq m$. By \ref{two closed 2}, we can find a set $S_j$ of size $kt-1$ such that both $G[S_j\cup \{u_j\}]$ and $G[S_j\cup \{v_j\}]$ contain a transversal $F$-factors and $S_1,\dots, S_j, W, \{u,v\}$ are all disjoint.
By repeating this for $j\in [t]\setminus\{i\}$, we obtain pairwise disjoint sets $S_1,\dots, S_t$ of each with size at most $kt-1$ such that $G[S_j \cup \{u_j\}]$ and $G[S_j \cup \{v_j\}]$ have a transversal $F$-factor for each $j \in [t] \setminus \{i\}$.

Let $S := \bigcup_{j \in [t] \setminus \{i\}} S_i \cup \{u_j, v_j \mid j \in [t] \setminus \{i\} \}$. Then it has at most $kt(t-1)$ vertices.
Moreover both $G[S\cup \{u\}]$ and $G[S\cup \{v\}]$ contain a transversal $F$-factors. The first one has a transversal $F$-factor $F_2 \cup \bigcup_{j\in [t]\setminus \{i\}} G[S_j\cup \{u_j\}]$ and the second one has a transversal $F$-factor $F_1\cup \bigcup_{j\in [t]\setminus \{i\}} G[S_j\cup \{v_j\}]$. This proves that $v$ is weakly $(F,m-2kt,k(t-1))$-reachable to $U_1$ as desired.
\end{proof}

\section{A proof of the main theorem}\label{sec:proof_of_main_2}

In order to prove \Cref{thm:main_theorem}, we collect some lemmas to carry out the three steps \ref{A1}--\ref{A3}.
For the first one \ref{A1}, the following lemma will be sufficient.
\begin{lemma}\label{lem:almost_spanning}
   Suppose $0< \alpha \ll 1/r, 1/\Delta, \varepsilon, \eta\leq 1$. 
    If 
   $G$ is an $r$-graph with $\delta_1(G) \geq \varepsilon n^{r-1}$ and  $\alpha^*(G) < \alpha n$, then it contains every $r$-forest $T$ with $\Delta(T) \leq \Delta$ and $|V(T)| \leq (1-\eta) n$.
\end{lemma}

To prove the bipartite version \Cref{thm:main_bipartite_version} we need the following strengthening.
\begin{lemma}\label{lem:almost_spanning_for_rainbow}
    Suppose $0< \alpha \ll 1/r, 1/\Delta, \varepsilon, \eta\leq 1$. 
    Let $G$ be an $n$-vertex $(r,1)$-graph with partition $X, Y$ where $|X| = \frac{(r-1)n+1}{r}$ and $|Y| = \frac{n-1}{r}$ and $\delta_1(G) \geq \varepsilon n^r$ and $\alpha^*_{\mathcal{X}}(G) < \alpha n$ where $\mathcal{X}=(X,X,\dots, X,Y)$ is an $(r+1)$-tuple.
  Suppose that $T$ is an $(r,1)$-hypertree on at most $(1-\eta)n$ vertices with $\Delta(T)\leq \Delta$ admitting a partition $A, B$ such that all the vertices in $B$ have degree $1$.
Then there exists an embedding $\varphi: T\rightarrow G$ such that $\varphi(A)\subseteq X$ and $\varphi(B)\subseteq Y$.
\end{lemma}


The following lemma encapsulates the step \ref{A2}.
\begin{lemma}\label{lem:embed_stars}
    Suppose $0<\alpha \ll \varepsilon, 1/r, 1/\Delta \leq 1$.
    Let $G$ be an $r$-partite $r$-graph with an $r$-partition $\mathcal{V}=(V_1, \ldots, V_r)$ with $|V_i|=n$ for each $2 \leq i \leq r$ and let $f:V_1 \rightarrow [\Delta]$ be a function such that $\sum_{v \in V_1} f(v) = n$.
Suppose that $G$ satisfies the following.
\begin{proplist}
\item $\alpha^*_{\mathcal{V}}(G) < \alpha n$.
\item $\delta_1(G[\mathcal{V}]) \geq \varepsilon n^{r-1}$.
\end{proplist}
    Then $G$ contains a collection of vertex-disjoint stars $\{S_v : v \in V_1\}$ such that $S_v$ is centered at $v$ and has $f(v)$ edges.
\end{lemma}

\begin{proof}[Proof of \Cref{lem:embed_stars}]
The lemma is clear if $G$ is a complete $r$-partite $r$-graph. Otherwise, we have $1\leq \alpha^*_{\mathcal{V}}(G) <\alpha n$, we have $n> 1/\alpha$, therefore we can choose constants $\gamma, \eta$ so that 
$$1/n<\alpha \ll \gamma \ll \eta \ll \varepsilon, 1/r\leq 1. $$
We replace each vertex $v \in V_1$ with $f(v)$ new vertices to obtain a new $r$-partite $r$-graph $G'$.
We abuse the notation to denote the set of new vertices as $V_1$ again. 
Then we have $\alpha^*_{\mathcal{V}}(G')\leq \Delta \alpha n$, $\delta_1(G'[\mathcal{V}]) \geq \varepsilon n^{r-1}$ and $|V_1|=\dots=|V_r|=n$.
It suffices to show that $G'$ contains a perfect matching which corresponds to a desired collection of vertex-disjoint stars. In doing this, we shall adopt the absorbing method.

We first construct many disjoint $F$-absorbers for every $\mathcal{V}$-balanced set of size $r$, where $F$ is a single edge. 
For this we prove the following claim.
\begin{claim}
    For any $\mathcal{V}$-balanced set $S= \{v_1,\dots, v_r\}$ and any set $U\subseteq V(G')\setminus S$ of size at most $r^2\eta n$ 
    and $i\in [r]$, there exist vertex-disjoint $i$-sets $e^{i}_1,\dots, e^{i}_r$ in $G'-U$ satisfying the following.
    \begin{proplist}
        \item  For each $j\in [r]$, $e^i_j$ crosses $(\{v_j\},V_{j+1},\dots, V_{j+i-1})$ where $V_{r+j'}= V_{j'}$ for each $j'\in [r]$. \label{eq: lattice 1}
        \item For each $i'\in [i]\setminus\{1\}$, $\bigcup_{j\in [r] }\left(e^i_j\cap V_{j+i'-1}\right)$ forms an edge in $G'-U$.\label{eq: lattice 2}
        \item For each $j\in [r]$, $d_{G'-U}(e^{i}_j)\geq \frac{1}{3^i}\varepsilon n^{r-i}$.\label{eq: lattice 3}
    \end{proplist}
\end{claim}
\begin{proof}
Let $G''=G'-U$.
For $i=1$, letting $e^1_j= \{v_j\}$ for each $j\in [r]$, the three properties are satisfied. Indeed, \ref{eq: lattice 1} is trivial and \ref{eq: lattice 2} is vacuous, and \ref{eq: lattice 3} holds as $\delta_1(G'')\geq \varepsilon n^{r-1} - |U|\cdot n^{r-2} \geq \frac{1}{3} \varepsilon n^{r-1}$.

Assuming the existence of $e^{i}_1,\dots, e^{i}_r$ for $i\geq 1$, we prove that we can obtain $e^{i+1}_1,\dots, e^{i+1}_r$.
For each $j\in [r]$, let $W_{j+i}\subseteq V_{j+i}$ be the set of verices $w$ such that $d_{G''}( e^i_j\cup \{w\}) \geq \frac{1}{3^{i+1}} \varepsilon n^{r-i-1}$. Then as $|V_i|= n$, Lemma~\ref{lem: find a set} implies that 
$$|W_{i+1}|\geq \frac{1}{n^{r-i-1}- \frac{1}{3^{i+1}} \varepsilon n^{r-i-1} +1 }(\tfrac{1}{3^i} \varepsilon n^{r-i} - (\tfrac{1}{3^{i+1} } \varepsilon n^{r-i-1}-1)n ) \geq \tfrac{1}{3^{i+1}} \varepsilon n.$$
As $i\leq r$, we have  $\alpha^*_{\mathcal{V}}(G'') < \Delta \alpha n <  \frac{1}{3^{i+1}} \varepsilon n$.
Thus, we can choose an edge $e'$ of $G''$ crossing $(W_1,\dots, W_r)$.
Let $e^{i+1}_j = e^{i}_j\cup (e'\cap W_{j+i})$ where $W_{r+j}=W_j$ for $j\in [r]$. Then \ref{eq: lattice 1} and \ref{eq: lattice 2} hold for $e^{i+1}_1,\dots, e^{i+1,r}$ by our choice of $e'$ and \ref{eq: lattice 3} holds by the definition of $W_j$. This proves the claim.
\end{proof}
We now show that for any $\mathcal{V}$-balanced set $S$ of size $r$, there are $\eta n$ pairwise disjoint $(F,r-1)$-absorbers for $S$. Take a maximum collection of pairwise disjoint $(F,r-1)$-absorbers for $S$ and 
let $U$ be the set of vertices in those absorbers. We may assume $|U|< r^2\eta n$, as otherwise we are done.

We apply the above claim to a given $\mathcal{V}$-balanced set $S$ for $i=r$, then \ref{eq: lattice 3} implies $d_{G'-U}(e^r_j)>0$ and it follows that each $e^r_j$ is an edge for each $j\in [r]$ as it is an $r$-set with positive degree.
Let $S'= \bigcup_{j\in [r]} e^r_j - S$. Then the edges coming from \ref{eq: lattice 2} form a perfect matching in $S'$ and the edges $e^r_1,\dots, e^r_r$ form a perfect matching in $S\cup S'$. Thus, $S'$ is an $(F,r-1)$-absorber of $S$. As one can find such an absorber disjoint from any set $U$ of size at most $r^2\eta n$, one can greedily take $\eta n$ pairwise disjoint $(F,r-1)$-absorbers of $S$. 
Hence, \Cref{lem:absorbing_using_template} implies that $G'$ contains a $\gamma$-absorbing set $A$ of size at most $\eta n$.

As $\alpha^*_{\mathcal{V}}(G'-A)\leq  \alpha^*_{\mathcal{V}}(G) \leq \Delta \alpha n$, we can greedily take pairwise disjoint edges to obtain a matching $M$ in $G'-A$ with $|V_i\setminus M| \leq \Delta \alpha n$.
As the set $U = V(G')\setminus (A\cup V(M))$ form a $\mathcal{V}$-balanced set of size at most $\Delta \alpha n< \gamma n$, $G'[A\cup U]$ contains a perfect matching $M'$. Then $M\cup M'$ is a perfect matching in $G'$, as desired.  

\end{proof}
Finally, the following lemma allows us to embed caterpillars. This lemma will be proved in Section~\ref{sec:cycle_factor}.
\begin{lemma}\label{lem:perfect_cycle_factor}
    Let $0<\alpha \ll \varepsilon, 1/r$ and $t \geq \frac{(100r)^r}{\varepsilon} $.
    Consider the linear cycle $C_{t}^{(r)}$ on the vertex set $[(r-1)t]$ where each edge is $[(r-1)(i-1),(r-1)i]$ for some $i\in [t]$ with $(r-1)t$ is also denoted by $0$.
    Suppose that an $r$-graph $G$ satisfies the following:
    \begin{proplist}
            \item  $G$ admits a partition $(\mathcal{V},C_{t}^{(r)})$ with $\mathcal{V}=(V_1,\dots, V_{(r-1)t})$ and $|V_i|=n$ for each $i\in [(r-1)t]$.
            \item For all $e\in E(C_{t}^{(r)})$ and $\mathcal{V}_{e}=(V_j: j\in e)$, we have $\alpha^*_{\mathcal{V}_{e}}(G)<\alpha n.$
        \item For every $e\in E(C_{t}^{(r)})$, we have $\delta(G[\mathcal{V}_{e}]) \geq \varepsilon n^{r-1}$ where $\mathcal{V}_{e} = (V_j : j\in e)$. \label{eq: M3}
    \end{proplist}
    Then $G$ contains a transversal $C_{t}^{(r)}$-factor.
\end{lemma}

The proof of these lemmas will be provided in later sections. We first prove \Cref{thm:main_theorem} assuming these lemmas.

\subsection{Proof of the main theorem}\label{sec:proof_of_main}
\begin{proof}[Proof of \Cref{thm:main_theorem}]
If $G$ is a complete $r$-graph, then the statement is trivial. Otherwise, we have $1<\alpha^*(G)<\alpha n$, thus $n> 1/\alpha$.
For given integers $r,\Delta$ and a real $\varepsilon>0$, we choose constants $\alpha,\beta,\gamma$ satisfying
$$0< 1/n < \alpha \ll \beta \ll \gamma  \ll 1/r, 1/\Delta, \varepsilon \leq 1.$$
 Consider an $n$-vertex $r$-graph $G$ that satisfies $\delta_1(G)\geq \varepsilon n^{r-1}$ and $\alpha^*(G)<\alpha n$ and $r-1$ divides $n-1$. 
 Consider an $n$-vertex $r$-hypertree $T$ with $\Delta(T)\leq \Delta$. We aim to prove that $G$ contains $T$ as a subgraph.
 We may assume that $G$ is not a complete graph.
In order to prove that $G$ contains $T$ as a subgraph, we consider two cases depending on the structure of $T$. For this case distinction, we let $t = \lceil\frac{(100r)^r}{\varepsilon} \rceil$.

By \Cref{lem:pendant_and_caterpillars}, either $T$ contains $\frac{n}{4r^2(t+1)\Delta}$ vertex-disjoint pendant stars or $\frac{n}{12r^2(t+1)\Delta}$ vertex-disjoint caterpillars each of length $t$. Note that pendant stars or caterpillars are all hypertrees with at most $r^2t\Delta$ vertices. As there are at most $2^{\binom{r^2t\Delta}{r}}$ 
isomorphism classes of $r$-hypertrees on at most $r^2t\Delta$ vertices, we can choose $m=\gamma n \leq 2^{-\binom{r^2t\Delta}{r}} \frac{n}{12r^2(t+1)\Delta}$ vertext-disjoint isomorphic subgraphs $S_1,\dots, S_m$ all of which are either non-trivial pendant stars or caterpillars of length $t$. We say that we are in Case 1 if $S_i$'s are all isomorphic non-trivial pendant stars and we are in Case 2 if $S_i$'s are all isomorphic caterpillars of length $t$. We define $r^*$ as follows. 
$$r^*=\left\{\begin{array}{ll}  r-1 & \text{ if Case 1 applies.} \\ 
(r-1)t-1 & \text{ if Case 2 applies.}\end{array}\right.$$
For each $S_i$, let $L_i$ be the set of edges in $S_i$ which are the leaf-edges of $T$ and let $C_i= S_i\setminus L_i$. Note that $C_i$ is either the center of the pendant star $S_i$ or the central path of the caterpillar $S_i$. 
Let $T'$ be the hyperforest obtained from $T$ by removing all edges of $S_i$ and deleting isolated vertices. 

Let $C'_i\subseteq V(C_i)\setminus V(T')$ be the set of vertices in $C_i$ which are incident with at least one leaf-edge in $L_i$.
Let $\ell:=\sum_{i\in [m]} |L_i|$ be the number of all leaf-edges in $S_1,\dots,S_m$.
Note that $\ell$ can be $0$ in Case 2, but it is not $0$ in Case 1, as $S_i$'s are non-trivial pendant stars in Case 1.

We assume that the vertices in each $C_i\setminus V(T')$ are ordered into $z^i_1,\dots, z^i_{r^*}$, so that all graphs $S_i$ are isomorphic via an isomorphism which preserves this ordering.
If each of $S_i$ contains at least one leaf-edge, we choose $i_*$ so that $z^i_{i_*}$ is incident with at least one leaf. If all $S_i$'s are caterpillars without any leaf-edge (hence $\bigcup_{i\in [m]} L_i=\emptyset$), then we leave $i_*$ to be undefined.

We take a partition of $V(G)$ into vertex sets $V_1, V_{2,1},\dots, V_{2,r^*}, V_{3,1},\dots, V_{3,r-1}$
satisfying the following.
\begin{proplist}
        \item \label{V1} $|V_1|=n_1:= |V(T')|+ \beta n$, $|V_{3,i}|= \ell$ and $|V_{2,i'}| \in  \{\lfloor \frac{n-n_1-(r-1)\ell}{r^*} \rfloor, \lceil \frac{n-n_1-(r-1)\ell}{r^*} \rceil \}$ for all $i\in [r-1]$ and $i'\in [r^*]$. \label{V1 property}
    \item \label{V2} For any $v\in V(G)$ and $X_1,\dots, X_{r-1} \in \{V_1,V_{2,1},\dots, V_{2,r^*}, V_{3,1},\dots, V_{3,r-1}\}$ with repetitions allowed, we have 
    $d_{G}(v, X_1,\dots, X_{r-1})\geq \frac{1}{2}\varepsilon \prod_{i\in [r-1]} |X_i|$. \label{V2 property}
\end{proplist}
Indeed, if \ref{V1 property} holds, then each set above has size either $0$ (which can happen in Case $2$ if $\ell=0$) or at least $m/r^*-\beta n > \gamma n/(2r^*)$. Hence, if we take a random partition of $V(G)$ satisfying \ref{V1 property}, then it is straightforward to check that Lemma~\ref{lem:general_concent} implies that \ref{V2 property} holds with probability $1-o(1)$. 
This shows the existence of a desired partition $V_1, V_{2,1},\dots, V_{2,r^*}, V_{3,1},\dots, V_{3,r-1}$.

Note that \ref{V2 property} implies $\delta(G[V_1]) \geq \frac{1}{2 r!}\varepsilon |V_1|^{r-1}$ and $G[V_1]$ has no $r$-partite hole of size $\alpha n \leq 2\alpha |V_1|$.
We have $|V(T')|\leq |V_1| - \beta n \leq (1-\beta) |V_1|$.
Thus \Cref{lem:almost_spanning} yields an embedding $\varphi: V(T') \rightarrow V_1$ of $T'$ into $G[V_1]$. 

Note that, as $|V(T)|=n$, the number of vertices in $\bigcup_{i\in [m]} (C_i\setminus V(T'))$ is $n-|V(T')|-(r-1)\ell = r^*m$. Thus we can distribute the vertices in $V_1\setminus \varphi(V(T'))$ to $V_{2,1},\dots, V_{2,r^*}$ to obtain new vertex sets $\widehat{V}_{2,1},\dots, \widehat{V}_{2,r^*}$ satisfying the following.
\begin{proplist}
        \item \label{widehatV1} $|\widehat{V}_{2,i}|= \frac{n-|V(T')|-(r-1)\ell}{r^*} = m$.
    \item  \label{widehatV2} For any $v\in V(G)$ and $X_1,\dots, X_{r-1} \in \{V_1,\widehat{V}_{2,1},\dots, \widehat{V}_{2,r^*}, V_{3,1},\dots, V_{3,r-1}\}$ with repetitions allowed, we have 
    $d_{G}(v, X_1,\dots, X_{r-1})\geq \frac{1}{3}\varepsilon \prod_{i\in [r-1]} |X_i|$.
\end{proplist}
Indeed, if we arbitrarily distribute them to make \ref{widehatV1} holds, then it is straightforward to check that \ref{widehatV2} follows from \ref{V2 property} and the fact that there are at most $\beta n$ additional vertices from $V_1\setminus \varphi(V(T'))$, where $\beta \ll \gamma, \varepsilon, 1/r^*$ because $|X_i|\geq m/2 \geq \gamma n/2$ for each $X_i$ in \ref{widehatV2}.

For each $S_i$, we let the vertices it has in common with $T'$ be $u_i$ in Case 1 and $u_i$ and $v_i$ in Case 2. In Case 1, the center $C_i$ of the pendant star $S_i$ is incident with $u_i$ and 
in Case 2, the central path $C_i$ of the caterplillar $S_i$ connects $u_i$ and $v_i$ in $T$. For our convenience, we also define $v_i$ in Case 1 by letting $v_i=u_i$. 

Let $R_1= \varphi(\{u_1,\dots, u_m\})$ and $R_2=\varphi(\{v_1,\dots, v_m\})$.
For each $i\in [r-1]$, let 
$$W_i=\left\{\begin{array}{ll}
   V_{3,i}  & \text{ in Case 1}. \\
   \widehat{V}_{2, r-1+i}  & \text{ in Case 2.} 
\end{array}\right.$$
We apply \Cref{lem: find matching} to the partition $(R_1, \widehat{V}_{2, 1}, \widehat{V}_{2,2},\dots, \widehat{V}_{2,r-1})$ and $(W_1,\dots, W_{r-1})$ playing the roles of $\mathcal{V}$ and $\mathcal{U}$, respectively to obtain a matching $M_1= M_1'\cup M''_1$.
In Case 2, we again apply \Cref{lem: find matching} to the partition $(R_2,\widehat{V}_{2, r^*-r+2}, ,\dots, \widehat{V}_{2,r^*})$ and $(W_1\setminus V(M''_1),\dots, W_{r-1}\setminus V(M''_1))$ playing the roles of $\mathcal{V}$ and $\mathcal{U}$, respectively to obtain a matching $M_2= M'_2\cup M''_2$. Indeed, \ref{widehatV2} together with the fact $\alpha \ll \varepsilon$ ensures that these applications are possible.
Taking $M'=M'_1\cup M'_2$ and $M''= M''_1\cup M''_2$, we have a matching $M=M'\cup M''$ satisfying the following.
\begin{proplist}
    \item $|M|=|R_1\cup R_2|$ and $|M''|\leq 2\alpha n$. \label{M1}
    \item each edge $e \in M'$ either crosses $(R_1, \widehat{V}_{2, 1}, \widehat{V}_{2,2},\dots, \widehat{V}_{2,r-1})$ or $(R_2,\widehat{V}_{2, r^*-r+2}, ,\dots, \widehat{V}_{2,r^*})$. \label{M2}
    \item each edge $e \in M''$ crosses $(R_1\cup R_2, W_1,\dots, W_{r-1})$. \label{M3}
\end{proplist}

Let $e_{1,i}$ be the edge in $M$ containing the vertex $u_i$ and let $e_{2,i}$ be the edge in $M$ containing the vertex $v_i$. Note that in Case 1, $e_{1,i}=e_{2,i}$.

In Case 1, $\varphi(T')\cup M$ yields a copy of the hypertree $T\setminus\bigcup_{i\in [m]}L_i$. Let $\varphi'$ be the embedding of  $T\setminus\bigcup_{i\in [m]}L_i$ into $\varphi(T')\cup M$ so that each vertex $z^i_j$ in $C_i$ is mapped to the vertex $e_{1,i}\cap \widehat{V}_{2, j}$. 

On the other hand, we need one more step in Case 2 before obtaining the desired partial embedding $\varphi'$ of $T\setminus\bigcup_{i\in [m]}L_i$. Assume we are in Case 2.
In Case 2, the edges $e_{1,i}$ and $e_{2,i}$ in $M$ correspond to the first and the last edges in the central paths in $S_i$'s. 
We wish to apply \Cref{lem:perfect_cycle_factor} to connect the edges in $M$ to obtain the tree $T\setminus\bigcup_{i\in [m]}L_i$. For this application of \Cref{lem:perfect_cycle_factor}, we need to modify the vertex partition. As there are some vertices in $\bigcup_{i\in [r-1] } (\widehat{V}_{2,i}\cup \widehat{V}_{2,r^*+1-i})$ not covered by $M$, we need to redistribute these vertices to other parts to obtain new pairwise disjoint vertex sets $V'_{r-1},\dots, V'_{r^*-r+2}$ which is compatible with \Cref{lem:perfect_cycle_factor}. Recall that each $W_i$ is $V_{2,r-1+i}$ in Case 2, hence some vertices in $V_{2,r-1+i}$ are covered by $M''$.
For each $i\in [r-1]$ and $2r-1\leq j\leq r^*+1-r$, let 
$$V'_i = (\widehat{V}_{2,i}\cup W_i)\cap V(M), \enspace V'_{r^*+1-i} = (\widehat{V}_{2,r^*+1-i}\cup W_i)\cap V(M) \text{ and } V'_{j} = \widehat{V}_{2,j}.$$
For each $i\in [r-1]$, let
$$V'_{r-1+ i} = (\widehat{V}_{2,r-1+i}\cup \widehat{V}_{2,i} \cup \widehat{V}_{2,r^*+1-i} )  \setminus V(M).$$
Then for each $i\in[r,2r-2]$, $V'_i$ is disjoint from $\varphi(T')\cup V(M)$ and $|V'_i|=m$ for all $i\in [r-1,r^*-r+2]$.

Let $x_i$ be the vertex in $e_{1,i}\cap V'_{r-1}$ and $y_i$ be the vertex in $e_{2,i}\cap V'_{r^*-r+2}$. We now identify $x_i$ and $y_i$ to obtain a new vertex $z_i$ and let $V''_{r-1}$ be the set of those vertices. For each $r\leq i\leq r^*-r+1$, let $V''_{i}=V'_i$.
Let $E(G')$ consist of the edges of $G- (V'_{r-1}\cup V'_{r^*-r+2})$ and 
the edges $(e\setminus x_i)\cup \{z_i\}$ for edges $e$ in $G[V'_{r-1},\dots, V'_{2r-2}]$ containing $x_i$ and the edges $(e'\setminus y_i)\cup \{z_i\}$ for edges $e'$ in $G[V'_{r^*-2r+3},\dots, V'_{r^*-r+2}]$ containing $y_i$.

Now we can check that this partition $V''_{r-1},V''_r,\dots, V''_{r^*-r+1}$  satisfies the conditions in \Cref{lem:perfect_cycle_factor}. To show this, consider $C_{r^*-2r+3}^{(r)}$ on the vertex set $[r^*-r+1]\setminus [r-2]$ with edges $[(r-1)i,(r-1)(i+1)]$ for each $i\in [r^*/(r-1)-2]$ where $(r-1)$ and $(r-1)(t-1)=r^*-r+1$ are identified.
For each $e\in C_{r^*-2r+3}^{(r)}$ and $i\in e$ and $v\in \mathcal{V}''_e$, \ref{V2} implies that $d(v,\widehat{\mathcal{V}}_{2,e-\{i\}})> \frac{1}{3}\varepsilon m^{r-1}$ where $\widehat{\mathcal{V}}_{2,e-\{i\}} = (V_{2,j}: j\in e-\{i\})$.
By \ref{M1} and the above definition, for each $V'_j$ with $r-1\leq j\leq r^*-r+2$, there exists $a_j$ such that $|V'_j\Delta \widehat{V}_{2,a_j}|\leq 2\alpha n$. 
Hence, for each $e\in C_{r^*-2r+3}^{(r)}$ and $i\in e$ and $v\in \mathcal{V}''_e$,
$d(v, \mathcal{V}''_{e-\{i\}}) 
\geq d(v, \widehat{\mathcal{V}}_{2,f }) - 2r \alpha n^{r-1} \geq \frac{1}{4}\varepsilon m^{r-1}$ where
$$f= \left\{ \begin{array}{ll}
\{a_j: j\in e-\{i\} \} & \text{ if } e\neq [(r-1)(t-2), (r-1)(t-1)]
 \\
\{a_j: j\in e-\{i\} - \{r^*-r+2\}\}\cup \{a_{r^*-r+2}\} & \text{ if } e= [(r-1)(t-2),(r-1)(t-1)].
\end{array} \right. $$
Here, $(r-1)(t-1) = r^*-r+2$.
This degree condition together with the fact $\alpha^*(G)< \alpha n \leq \alpha^{1/2} m$ implies that we can apply \Cref{lem:perfect_cycle_factor} to obtain a cycle factor.
Such a cycle factor corresponds to a desired linkage between $V'_{r-1}$ and $V'_{r^*-r+2}$ connecting each $x_i$ with $y_i$ using all vertices in $\bigcup_{r-1\leq i\leq r^*-r+2} V'_{i}$.
This together with $\varphi(T')\cup M$ yields a copy of $T\setminus\bigcup_{i\in [m]} L_i$ in $G[V_1\cup V_2]$. 

If $\bigcup_{i\in [m]} L_i =\emptyset$, this finishes the proof of the theorem. Otherwise, recall that the number $i_*$ is defined. Now, in both Case 1 and Case 2, we have an embedding $\varphi'$ of $T''= T\setminus\bigcup_{i\in [m]} L_i$.
In both cases, let $U_0$ be the image of the vertices in $\varphi'(\bigcup_{i\in [m]} C'_i)$.

In Case 1, let $U_i = (\widehat{V}_{2,i}\cup V_{3,i})\setminus V(M)$ for each $i\in [r-1]$. In Case 2, let $U_i = V_{3,i}$ for each $i\in [r-1]$. Then in either case, $|U_i\Delta V_{3,i}|\leq 2\alpha n$. Hence for all $v\in V(G)$,
$$d(v,U_1,\dots, U_{r-1}) \geq \frac{1}{2}\varepsilon \prod_{i\in [r-1]}|U_i| - r\alpha n^{r-1} \geq \frac{1}{3}\varepsilon \prod_{i\in [r-1]}|U_i|.$$
Also, by our definition of $i^*$ and $U_0$, we have $\widehat{V}_{2,i^*} \subseteq U_0$, thus for any vertex $v\in U_i$ with $i\in [r-1]$,
$$d(v, U_0,\dots ,U_{i-1}, U_{i+1},\dots, U_{r-1})\geq d(v, \widehat{V}_{2,i^*},V_{3,1},\dots, V_{3,i-1}, V_{3,i+1}\dots, V_{3,r-1}) - r\alpha n^{r-1} \geq \gamma |U_0| \ell^{r-1}.$$

Let $f$ be a function on $U_0$ where $f(u)$ is the number of leaf-edges incident with the corresponding vertex $\varphi'^{-1}(u)$ in $T$.
Then $\sum_{u\in U_0} f(u) =\ell = |U_i|$ for all $i\in [r-1]$.
Combined with the condition $\alpha^*(G)<\alpha n < \beta \ell$, \Cref{lem:embed_stars} yields the desired stars centered at each vertex in $U_0$, yielding a copy of $T$ in $G$. This proves the theorem.
\end{proof}
The proofs of \Cref{thm:main_theorem_cycle} and \Cref{thm:main_bipartite_version} are very similar to the proof of \Cref{thm:main_theorem}, so we just briefly mention differences without providing a formal proof. 

For \Cref{thm:main_theorem_cycle}, we can find many vertex-disjoint bare paths whose deletion from a loose Hamilton cycle leaves a linear forest. Hence, it is strightforward to follow the proof of \Cref{thm:main_theorem} to deduce \Cref{thm:main_theorem_cycle}. 
For \Cref{thm:main_bipartite_version}, we can replace the application of \Cref{lem:almost_spanning} by \Cref{lem:almost_spanning_for_rainbow} in the proof. Furthermore, when we take a random partition of vertices, we partition both $A$ and $B$ with appropriate sizes. With these differences in our mind, it is again straightforward to follow the proof of \Cref{thm:main_theorem} to deduce \Cref{thm:main_bipartite_version}.

\section{Finding almost spanning hypertrees}\label{sec:almost_spanning}

In this section, we prove \Cref{lem:almost_spanning} and \Cref{lem:almost_spanning_for_rainbow}.
As it is easy to see that \Cref{lem:almost_spanning_for_rainbow} implies \Cref{lem:almost_spanning}, we only need to prove \Cref{lem:almost_spanning_for_rainbow}.

\begin{proof}[Proof of \Cref{lem:almost_spanning_for_rainbow}]
As the lemma is clear if $G$ is a complete $(r,1)$-graph, we may assume $1<\alpha^{*}_{\mathcal{X}}(G)<\alpha n$. Hence we have $1/n < \alpha$.
Choose additional constants $\mu, \beta, \gamma$ so that
$$0< 1/n<\alpha \ll \mu \ll \beta \ll \gamma \ll \varepsilon, 1/r, 1/\Delta, \eta\leq 1.$$
By applying \Cref{lem:tree_split} with  constant $\mu>0$ as above, we obtain a sequence of hyperforests $T_0 \subseteq T_1 \subseteq \ldots \subseteq T_\ell = T$ and an exceptional index $s \in [\ell]$ satisfying \ref{it:T0}--\ref{it:path}.

Let $n':=|V(T)| \leq (1-\eta)n$. As it is an $(r+1)$-hypertree, it contains $\frac{n'-1}{r}$ edges.  Hence $|B|= \frac{n'-1}{r} \leq (1-\eta/2)\frac{n}{r}$ and 
$|A|= n'-|B| = \frac{(r-1)n'+1}{r} \leq (1-\eta/2) \frac{(r-1)n}{r}$.

We aim to iteratively embed $T$ into $G$. For this, we first take a subset of $X\cup Y$ satisfying the following properties.

\begin{claim}\label{claim:random_set_in_embedding_almost_spanning}
There exists a subset $R\subseteq X\cup Y$ satisfying the following, where $\mathcal{R}=(R,\dots, R)$ is an $r$-tuple.
\begin{proplist}
\item $|R| = 2\gamma n$ and $\min\{ |R\cap X|, |R\cap Y|\} \geq \gamma n/(2r)$. \label{Rprop1}
    \item For every $v\in X \cup Y$, $d_{G}(v,\mathcal{R}) \geq 10 r \beta n^r$. \label{Rprop2}
    \item For every $u,v\in X$, there exist at least $\beta n$ internally vertex-disjoint paths of length three between $u$ and $v$ whose internal vertices are all contained in $R$. \label{Rprop3}
\end{proplist}
\end{claim}
    \begin{proof}
    Consider a random set $R\subseteq V(G)$ of size $2\gamma n$ chosen uniformly at random. By Chernoof's bound (Lemma~\ref{lem:chernoff_hypergeometric}), \ref{Rprop1} holds with probability at least $1-o(1)$.
    
Apply \Cref{lem:general_concent} to $G$ with the random partition $(R,V(G)\setminus R)$, with the probability $1-o(1)$, every $v\in X$ satisfies 
$$d_G(v,\mathcal{R})\geq \frac{1}{2}\varepsilon |R|^r.$$
Hence, the union bound yields a choice of $R$ satisfying \ref{Rprop1} and \ref{Rprop2} as $|R|=2\gamma n$ implies $\frac{1}{2}\varepsilon |R|^r\geq 10r\beta n^{r}$. Take such a choice of $R$.

To show that $R$ satisfies \ref{Rprop3}, consider arbitrary two vertices $u,v \in X$.
We claim that there are at least $\beta n$ internally disjoint paths of length three in $G$ whose internal vertices all lie in $R$.
Greedily choose vertex-disjoint stars $S_u$ and $S_v$ of size $4 \beta n$ centered at $u$ and $v$, respectively, where each edge of them crosses $(\{u\}, R, \dots, R)$. 
Indeed, it is possible as \ref{Rprop2} implies that, for any such vertex-disjoint stars $S'_u$ and $S'_v$ of smaller sizes, the vertex $w\in \{u,v\}$ satisfies $d_{G-S'}(w, R, \dots, R) \geq 10r\beta n^{r} - |S'|n^{r-1} > 10 r \beta  n^{r} - 8r \beta n^{r-1} > 0$ where $S'= V(S'_u)\cup V(S'_v) - \{u,v\}$.
For each edge in $S_u$, we choose one vertex which is not $u$ to form a vertex set $V_1$. Similarly, choose one vertex from each $S_v$ to form a vertex set $V_r$.
Take arbitrary pairwise disjoint vertex sets $V_2,\dots, V_r$ of size $4 \beta n$ in $R$. We apply Lemma~\ref{lem: find matching} to find a matching of size $4\beta n - \alpha n  \geq \beta n$ where each edge crosses $(V_1,\dots, V_r)$. This provides at least $\beta n$ internally-disjoint paths of length $3$ between $u$ and $v$ whose internal vertices all belong to $R$.
    \end{proof}

Now we iteratively embed the hyperforests $T_0,T_1,\dots, T_{\ell}$ where each embedding extends the previous one. \newline 
    
\noindent   \textbf{Embedding $T_0$:} 
    We first embed $T_0$ into $G - R$. As $|R|\leq 2\gamma n$, we have $\delta_1(G-R) \geq \varepsilon n^{r}/2$. As $|V(T_0)| \leq r \mu  n$ and $\mu\ll \varepsilon, 1/r$, we can greedily embed $T_0$ into $G-R$. Furthermore, we can ensure that all the vertices in $B$ are embedded into the set $Y$ while all the vertices  in $A$ are embedded into the set $X$.
    Let $\varphi_0:V(T_0) \rightarrow V(G-R)$ be such an embedding. \newline 
    
\noindent    \textbf{Embedding $T_i$ for $1 \leq i \leq s$.} 
Assume that we have an embedding $\varphi_{i-1}$ of $T_{i-1}$ with $|\varphi_{i-1}(T_{i-1})\cap R| \leq (i-1)r\alpha n$ such that all the vertices in $B$ are embedded into the set $Y$ while all the vertices  in $A$ are embedded into the set $X$.

From \ref{it:matching}, $T_{i}-T_{i-1}$ consists of a matching $M$. Let $C_i$ be the set of vertices in $T_{i-1}$ which belong to $M$ and let $D_i$ be the vertex set $V(T_{i})\setminus V(T_{i-1})$. Then $|D_i|= r|C_i|$. Let $m:=|C_i|$.
Note that every vertex in $C_i$ has degree at least two, hence it belongs to $A$ while each edge in $M$ contains exactly one vertex in $B$.
As $\varphi_{i-1}$ embeds every vertex in $A\cap V(T_{i-1})$ to $X$ and every vertex in $B\cap V(T_{i-1})$ to $Y$, we have 
$$|X\setminus (\varphi_{i-1}(T_{i-1})\cup R)| \geq  |A\setminus V(T_{i-1})|  + |X|-|A| -|R|\geq |D_i\setminus B|+ \frac{(r-1) \eta n}{r} - 2\gamma n \geq (r-1) m + (r-1)\alpha n, \text{ and }$$
$$|Y\setminus (\varphi_{i-1}(T_{i-1})\cup R)| \geq |B\setminus V(T_{i-1})|+ |Y|-|B| -|R| \geq  |D_i\cap B| +  \frac{\eta n}{r} -2\gamma n \geq m + \alpha  n.$$
Hence, we may choose pairwise vertex-disjoint sets $V_1,\dots, V_{r-1} \subseteq X\setminus \varphi_{i-1}(T_{i-1})$ of size $m$ and a set $V_{r}\subseteq Y\setminus \varphi_{i-1}(T_{i-1})$ of size $m$.
Letting $V_0 = \varphi_{i-1}(C_i)$ and $\mathcal{V}=(V_0,\dots, V_r)$, the condition $\alpha^*_{\mathcal{X}}(G)<\alpha n$ implies $\alpha^*_{\mathcal{V}}(G)<\alpha n$.

As $|R\cap \varphi_{i-1}(T_{i-1})|\leq (i-1)r\alpha n$, \ref{Rprop2} implies that for any $v\in V_0$ and $R'= R\setminus \varphi_{i-1}(T_{i-1})$, we have 
$d_G(v,R',\dots, R') \geq 10r\beta n^{r} - (r-1)!\cdot (i-1)r\alpha n^{r} \geq \beta n^r$ as $i\leq \ell\le 10^5r\Delta(T)\mu^{-2}$ and $\alpha\ll \mu,\beta, \frac{1}{r}, \frac{1}{\Delta}$.
Hence, we can apply \Cref{lem: find matching} with the partitions $\mathcal{V}$ and $\mathcal{U}= (R',\dots, R')$ to obtain a matching $M'$ covering $V_0= \varphi_{i-1}(C_i)$ such that $M'-V_0$ lies outside $\varphi_{i-1}(T_{i-1})$ and
$M'$ intersects $R$ in at most $r \alpha n$ vertices.
By mapping the edges in $M$ to $M'$, we extend $\varphi_{i-1}$ to $\varphi_i$ which embeds $T_i$ to $G$ which uses at most $(i-1)r\alpha n + r\alpha n \leq ir\alpha n$ vertices in $R$. 

\noindent     \textbf{Embedding $T_{s+1}$.}
    For $i=s+1$, we need to embed paths instead of a matching.
    Assume that we have an embedding $\varphi_{s}$ of $T_{s}$ with $|\varphi_{s}(T_{s})\cap R| \leq sr\alpha n$ such that all the vertices in $B\cap V(T_s)$ are embedded into $Y$ while all the vertices  in $A\cap V(T_s)$ are embedded into $X$.

By \ref{it:path}, there are $u_1,\dots, u_m$ and $v_1,\dots, v_m$ with $m\leq \mu n$ where $T_{s+1}$ is obtained from $T_s$ by adding vertex-disjoint bare paths $P_1,\dots, P_m$ of length three where each $P_i$ connects $u_i$ and $v_i$.
By \ref{Rprop3}, for each $i\in [m]$, there are at least $\beta n$ internally vertex-disjoint paths of length three from $\varphi_s(u_i)$ to $\varphi_s(v_i)$ whose internal vertices are all contained in $R$.
We greedily choose vertex-disjoint paths $Q_i$ of length three connecting $\varphi_s(u_i)$ and $\varphi_s(v_i)$ through $R\setminus \varphi_s(T_s)$.
Indeed, for each $i\in [m]$, as $|R\cap \varphi_s(T_s)|\leq s r\alpha n$ and $|\bigcup_{j<i}V(Q_i)|<3r\mu n$ and $\alpha,\mu\ll \beta, 1/r$, at least one of those $\beta n$ internally-disjoint paths of length three from $\varphi_s(u_i)$ to $\varphi_s(v_i)$ through $R$ is disjoint from $Q_1,\dots, Q_{i-1}$ and $R\cap \varphi_s(T_s)$. 
By embedding each $P_i$ to $Q_i$ in such a way that the vertices in $X$ embeds into the vertices of $X$ (which is possible as every edge in $T$ and $G$ contains exactly one vertex in $B$ and $Y$, respectively), we obtain a desired embedding $\varphi_{s+1}$ of $T_{s+1}$ with $|\varphi_{s+1}(T_{s+1})\cap R| \leq sr\alpha n + 3r\mu n$ such that all the vertices in $B$ are embedded into the set $Y$ while all the vertices  in $A$ are embedded into the set $X$.

\noindent    \textbf{Embedding $T_i$ for $s+2\le i\le \ell$.} We follow the same arguments as in the cases $1\le i\le s$ to extend each embedding $\varphi_{i-1}$ of $T_{i-1}$ to $\varphi_i$ of $T_i$ by adding matchings. The only difference is that each $\varphi_i$ uses at most $(i-1)r\alpha n + 3r\mu n \leq \beta n/10$ vertices from $R$. By repeating the above process, we obtain a desired embedding $\varphi_{\ell}$ of $T$ into $G$.
\end{proof}

Note that the same argument still works for \Cref{lem:almost_spanning}. The lacks of the partitions $X\cup Y$ and $A\cup B$ only make the proof simpler.

\section{Finding a cycle factor}\label{sec:cycle_factor}

We now prove \Cref{lem:perfect_cycle_factor}. 
If every $G[\mathcal{V}_e]$ is a complete $r$-partite $r$-graph for $e\in E(C_{t}^{(r)})$, then the conclusion is trivial. Otherwise, we have $1\leq \alpha_{\mathcal{V}_e}(G)\leq \alpha n$, thus $n > 1/\alpha$. We choose additional constants $\gamma,\eta, \beta, \delta, T$ such that 
$$ 0< 1/n <  \alpha \ll 1/T \ll \gamma\ll \eta \ll \beta  \ll \delta \ll \varepsilon, 1/r, 1/t, 1/k \leq 1.$$
Let $F=C_{t}^{(r)}$ and recall that the edges of $F$ are $e(i):= [(r-1)(i-1),(r-1)i]$ for $i\in [t]$ where $(r-1)t$ and $0$ are identified.

We first prove the following claim, which is useful later for us to close the paths into cycles.
\begin{claim}\label{cl: short connection}
Let $i'\in [t]$ and  $i\in [(r-1)i', (r-1)(i'+1)-1]$. 
Let $x,y\in V_i$ be two (not necessarily distinct) vertices and let $H$ be a $(2r-2)$-graph where $E(H)$ consists of sets $A$ such that both $G[A\cup \{x\}]$ and $G[A\cup\{y\}]$ contains a $\mathcal{V}_{e(i')\cup e(i'+1)}$-transversal path of length two.
Let $U_i = \{x,y\}$ and for each $j\in e(i')\cup e(i'+1)-\{i\}$, let $U_j\subseteq V_j$ be a subset such that 
$H[\mathcal{U}_{e(i')\cup e(i'+1)-\{i\}}]$ contains at least $\delta^{2r} n^{2r-2}$ edges where $\mathcal{U}_{e(i')\cup e(i'+1)-\{i\}}  = (U_j: j\in e(i')\cup e(i'+1)-\{i\})$.
Then there exist subsets $X\subseteq U_{(r-1)(i'-1)}$ and $Y\subseteq U_{(r-1)(i'+1)}$ of size at least $\frac{1}{2}\delta^{2r} n$ such that the following holds.
\begin{proplist}
\item    For any $u\in X$ and $v\in Y$, there exist (not necessarily distinct) edges $e_1,e'_1 \in G[\mathcal{U}_{e(i')}]$ and $e_2,e'_2 \in G[\mathcal{U}_{e(i'+1)}]$ such that
    both $e_1\cup e_2$ and $e'_1\cup e'_2$ are $\mathcal{V}_{e(i')\cup e(i'+1)}$-transversal paths of length two connecting $u$ and $v$, and $(e_1\cup e_2) \triangle (e'_1\cup e'_2) = \{x\}\triangle\{y\}$.  \label{eq: short path}
\end{proplist}
\end{claim}
\begin{proof}
If $i=(r-1)i'$, then it has degree two in $F$. \Cref{lem: find a set} implies that for $j\in \{(r-1)(i'-1), (r-1)(i'+1)\}$, the number of the vertices $w$ in $U_j$ which belong to at least one edge of $H$ is at least $\frac{1 }{n^{2r-3}}\cdot \delta^{2r} n^{2r-2} = \delta^{2r} n$.
Hence, we may choose two sets $X\subseteq U_{(r-1)(i'-1)}, Y\subseteq U_{(r-1)(i'+1)}$ of at least $\delta^{2r} n$ vertices each of which lies in an edge in $H$. We can check that \ref{eq: short path} holds with this choice. Indeed, if $u\in X$ and $v\in Y$, then the definition of $X$ ensures two edges $e_1, e'_1$ in $G[\mathcal{V}_{e(i')}]$ with $e_1\Delta e'_1 = \{x\}\triangle\{y\}$. Similarly, the definition of $Y$ ensures two edges $e_2, e'_2$ in $G[\mathcal{V}_{e(i'+1)}]$ with $e_2\triangle e'_2 =\{x\}\triangle\{y\}$ that yield \ref{eq: short path}.

If $i> (r-1)i'$, then it has degree one in $F$, and \Cref{lem: find a set} implies that
$U_{(r-1)i'}$ contains at least $ \frac{1}{n^{2r-3}} (\delta^{2r} n^{2r-2} -  \frac{1}{2}\delta^{2r} n^{2r-2} )$ vertices $z$ with $d_H(z)\geq \frac{1}{2}\delta^{2r} n^{2r-3}$.
Fix such a choice $z\in U_{(r-1)i'}$, and let $H_z$ be the $(2r-3)$-graph with the edge set $N_H(z)$.
Then \Cref{lem: find a set} implies that for each $j\in \{(r-1)(i'-1), (r-1)(i'+1)\}$, $U_j$ contains at least $\frac{1}{n^{2r-4}}\cdot \frac{1}{2}\delta^{2r} n^{2r-3} \geq \frac{1}{2}\delta^{2r} n$ vertices $w$ with $d_{H_z}(w)\geq 1$.
Thus, we choose a set $X\subseteq U_{(r-1)(i'-1)}$ and $Y\subseteq U_{(r-1)(i'+1)}$ of such vertices, each of size at least $\frac{1}{2}\delta^{2r} n$. We can check that \ref{eq: short path} holds with this choice. Indeed, if $u\in X$ and $v\in Y$, then the definition of $X$ ensures at least one edge $e_1$ containing $\{z,u\}$ in $G[\mathcal{U}_{e(i')}]$. Moreover, the definition of $Y$ ensures two edges $e_2, e'_2$ containing $z$ in $G[\mathcal{U}_{e(i'+1)}]$ with $e_2\triangle e'_2 = \{x, y\}$. Thus, taking $e'_1=e_1$, we have \ref{eq: short path}.
\end{proof}

In order to find a transversal $F$-factor, we first need to take an absorbing set.
For this purpose, we first prove the following claim.
\begin{claim}
For each $i\in [(r-1)t]$, there exist $k_i\leq \delta^{-1}$ and pairwise disjoint sets $U_i^1,\dots, U_i^{k_i}\subseteq V_i$ of size at least $\delta n$ each of which is $(F,\beta n,2^{10\varepsilon^{-4}} )$-closed and $|V_i\setminus \bigcup_{j\in [k_i]} U_i^j|\leq \delta n$. 
\end{claim}
\begin{proof}
Let $c= \lceil 2\varepsilon^{-2} \rceil$. By \Cref{lem:partition_into_closed}, it suffices to show that every set $Z$ of any $c+1$ vertices in $V_i$ contains two distinct vertices $x,y$ that are $(F,\varepsilon^4 n/(100r), 1)$-reachable.
When we deal with indices, we take those indices modulo $(r-1)t$ so that $V_0=V_{(r-1)t}$.

Let $i$ be an index with $(r-1)i'\leq i < (r-1)(i'+1)$ for some $i'\in[t]$ and consider two edges $e(i')$ and $e(i'+1)$ of $F$.
Take a set $Z$ of $c+1$ vertices in $V_i$.
First, we choose two distinct vertices $x,y\in Z$ such that they are $(F,\varepsilon^4 n/(100r), 1)$-reachable.

Let $H$ be the $(2r-1)$-graph whose edges are the vertex sets $A$ each of which forms a two-edge path in $G[\mathcal{V}_{e(i')\cup e(i'+1)}]$ and $|A\cap Z|=1$.
As the vertices in $Z$ have degree at least $\varepsilon n^{r-1}$ in $G[\mathcal{V}_{e(i')\cup e(i'+1)}]$, they have degree at least $\varepsilon^2 n^{2r-2}$ in $H$. Hence 
the inclusion-exclusion principle implies 
$$n^{2r-2} \geq  \sum_{x\in Z} d_{H}(x) - \sum_{x\neq y\in Z} |N_{H}(x)\cap N_{H}(y)| \geq (c+1) \varepsilon^2 n^{2r-2} -  \sum_{x\neq y\in Z} |N_{H}(x)\cap N_{H}(y)|,$$
hence there exist $x\neq y\in Z$ with $|N_H(x)\cap N_H(y)|\geq \binom{c+1}{2}^{-1} n^{2r-2} \geq \frac{1}{4}\varepsilon^4 n^{2r-2}$.

In order to show that $x,y$ are $(F,\varepsilon^4 n/(100r), 1)$-reachable, consider a set $W\subseteq V(G)$ with $|W\cap V_j|\leq \varepsilon^4 n/(100r)$ for each $j\in [r]$. Let $U_i= V_i\setminus W$. 
Let $H'$ be the $(2r-2)$-graph whose edges are the sets in $N_{H}(x)\cap N_{H}(y)$ which do not intersect $W$. 
Then $H'$ contains at least $\frac{1}{4}\varepsilon^4 n^{2r-2} - 2r\cdot \frac{1}{100r}\varepsilon^4 n^{2r-2} \geq \frac{1}{5}\varepsilon^4 n^{2r-2}$ edges.
Thus we may apply \Cref{cl: short connection} to obtain two sets $W_{(r-1)(i'-1)}\subseteq V_{(r-1)(i'-1)}\setminus W$ and $W_{(r-1)(i'+1)}\subseteq V_{(r-1)(i'+1)} \setminus W$ of size at least $\frac{1}{2}\delta^{2r} n$ satisfying \ref{eq: short path}. 
Now, for all $j < (r-1)(i'-1)$ or $j >(r-1)(i'+1)$ we take an arbitrary subset $W_{j}\subseteq (V_j \setminus W)$ such that $|W_j| =  \frac{1}{2}\delta^{2r} n$.
Let $\mathcal{W}$ be the tuple $(W_{(r-1)(i'+1)}, \dots , W_{(r-1)t}, W_1,\dots, W_{(r-1)(i'-1)})$. 
We apply \Cref{lem: find path} to $G$ and $\mathcal{W}$ to find a path $P$ connecting a vertex $u$ in $W_{(r-1)(i'-1)}$ and $v$ in $W_{(r-1)(i'+1)}$. 
For these choices of $u$ and $v$,  \ref{eq: short path} yields edges $e_1,e'_1, e_2$ and $e'_2$. The path $P$ together with $e_1,e_2$ gives a cycle $C$ and the path $P$ together with $e'_1,e'_2$ gives another cycle $C'$ such that both of them are transveral copies of $F$.
Moreover, $V(C)\triangle V(C') = \{x,y\}$. This shows that $x,y$ are $(F,\varepsilon^4 n/(100r), 1)$-reachable as desired.
    \end{proof}

We now prove the following claim.
  \begin{claim}
        For each $i \in [(r-1)t]$, the set $V_i$ is $(F, \beta n/2, \delta^{-1})$-closed.
    \end{claim}
    \begin{proof}
We shall apply \Cref{lem:merge_closed_and_nonclosed_sets} to obtain the desired closedness.
For this, fix an index $i\in [(r-1)t]$, a vertex $v\in V_i\setminus U_{i}^{1}$ and a vertex set $W\subseteq V(G)\setminus\{v\}$ such that $|W\cap V_j|\leq \beta n/2$ for all $j\in [(r-1)t]$. We aim to find two vertex-disjoint copies $F_1$ and $F_2$ of $F$ satisfying \ref{two closed 1} and \ref{two closed 2}.
Let $i'\in [t]$ be such that $(r-1)i'\leq i < (r-1)(i'+1)$.
Consider edges $e(i'-1), e(i'), e(i'+1), e(i'+2) \in E(F)$.

As $\delta_1(G)\geq \varepsilon n^{r-1}$, the vertex $v$ belongs to at least $\varepsilon^2 n^{2r-2}-2r\cdot \tfrac{\beta}{2} n^{2r-2}\ge \tfrac{1}{2}\varepsilon^2 n^{2r-2}$ $\mathcal{V}_{e(i')\cup e(i'+1)}$-transversal paths of length two in $G[\mathcal{V}_{e(i')\cup e(i'+1)}]-W$. 
By the pigeonhole principle, without loss of generality, we may assume that at least $(\prod_{j\in e(i')\cup e(i'+1)} k_j)^{-1} \tfrac{1}{2}\varepsilon^2 n^{2r-2} \geq \delta^{2r} n^{2r-2}$ such paths are crossing $\mathcal{U}^{1}_v:= ( U^{1}_{(r-1)(i'-1)}\setminus W,\dots, U^{1}_{i-1}\setminus W, \{v\}, U^{1}_{i+1}\setminus W,\dots, U^{1}_{(r-1)(i'+1)}\setminus W)$. 
Let $H$ be a $(2r-2)$-graph on vertex partition $(U^{1}_{(r-1)(i'-1)}\setminus W,\dots, U^{1}_{i-1}\setminus W, U^{1}_{i+1}\setminus W,\dots, U^{1}_{(r-1)(i'+1)}\setminus W)$ where $E(H)$ consists of all sets $A$ such that $G[A\cup \{v\}]$ contains such a path crossing $\mathcal{U}^{1}_v$. Then $H$ contains at least $\delta^{2r} n^{2r-2}$ edges.
By \Cref{cl: short connection}, there exist $X\subseteq U^1_{(r-1)(i'-1)}\setminus W$ and $Y\subseteq U^{1}_{(r-1)(i'+1)}\setminus W$ such that \ref{eq: short path} holds with $v$ playing the role of $x=y$ and $U^{1}_j\setminus W$ playing the role of $U_j$ for every $j\in e(i')\cup e(i'+1)\setminus\{i\}$.

Now, every vertex in $X$ and $Y$ has degree at least $\varepsilon n^{r-1} - |W|\cdot rn^{r-2} \geq \frac{1}{2}\varepsilon n^{r-1}$ in $G[\mathcal{V}_{e(i'-1)}]-W$ and $G[\mathcal{V}_{e(i'+2)}]- W$, respectively. 
Hence, by the pigeonhole principle, without loss of generality, we assume that at least $(\prod_{j\in e(i'-1)} k_j)^{-1} \varepsilon^2 n^{r-1}|X| \geq \delta^{4r} n^{r}$ edges of $G-W$ are crossing $\mathcal{U}^{1}_{e(i'-1)}:= ( U^{1}_{(r-1)(i'-2)},\dots, U^{1}_{(r-1)(i'-1)-1}, X)$ and 
at least $\delta^{4r} n^{r}$ edges of $G-W$ are crossing $\mathcal{U}^{1}_{e(i'+2)}:= (Y, U^{1}_{(r-1)(i'+1)+1},\dots, U^{1}_{(r-1)(i'+2)})$.
Using \Cref{lem: find a set} with $\emptyset$ playing the role of $U$, we can find subsets $X'\subseteq X$ and $Y'\subseteq Y$ of size at least $\frac{1}{ n^{r-1}}\cdot \frac{1}{2}\delta^{4r} n^{r} \geq \frac{1}{4} \delta^{4r} n$ which consist of all vertices having degree at least $\frac{1}{2}\delta^{4r} n^{r-1}$ in $G[\mathcal{U}^{1}_{e(i'-1)}]-W$ and $G[\mathcal{U}^{1}_{e(i'+2)}]-W$, respectively.

\noindent \textbf{Find a copy $F_2$ of $F$ containing $v$.}  
We now apply \Cref{lem: find path} to find a path $P$ in $G- W$ connecting a vertex $y\in Y$ and a vertex $x\in X$ which crosses $(Y, U^{1}_{(r-1)(i'+1)+1\setminus W},\dots, U^{1}_{(r-1)(i'-1)-1}\setminus W, X)$. This is possible as each set in the tuple has size at least $\frac{1}{2}\delta^{2r}n-\beta n > (r-1)t \alpha n$.
As $X,Y$ satisfies \ref{eq: short path}, there exists a two-edge path from some $x\in X$ to some $y\in Y$ containing $v$, crossing $(X,U^{1}_{(r-1)(i'-1)+1}\setminus W,\dots, U^{1}_{i-1},\{v\},U^{1}_{i+1}\setminus W,\dots, Y)$. This together with $P$ yields a desired copy $F_2$ of $F$ that crosses $(U^1_1\setminus W,\dots,U^1_{i-1}, \{v\}, U^1_{i+1}, \dots, U^1_{(r-1)t}\setminus W)$.

\noindent \textbf{Find a desired copy $F_1$ of $F$ intersecting $U^1_{i}$.} 
We similarly apply \Cref{lem: find path} to find a path $P'$ from a vertex $x'\in X'$ to $y'\in Y'$ that crosses $(X', U^{1}_{(r-1)(i'-1)+1}\setminus W,\dots, U^{1}_{(r-1)(i'+1)-1}\setminus W, Y')$ and is disjoint from $V(F_2)$. This is possible as each set in the tuple has size at least $\frac{1}{4}\delta^{4r}n - \beta n >  (r-1)t \alpha n$.
Recall that $x'\in X'$ and $y'\in Y'$ have degree at least $\frac{1}{2}\delta^{4r} n^{r-1}$ in $G[\mathcal{U}^{1}_{e(i'-1)}]-W$ and $G[\mathcal{U}^{1}_{e(i'+2)}]-W$, respectively. Applying \Cref{lem: find a set} to $N_{G[\mathcal{U}^1_{e(i'-1)}]-W}(x')$ and $N_{G[\mathcal{U}^1_e(i'+2)]-W}(y')$, we can find two sets $X''\subseteq U^{1}_{(r-1)(i'-2)}\setminus (W\cup V(F_2))$ and $Y''\subseteq U^{1}_{(r-1)(i'+2))}\setminus (W\cup V(F_2))$ of at least $\tfrac{1}{2}\delta^{4r}n$ vertices that belong to at least one edge in $N_{G[\mathcal{U}^1_{e(i'-1)}]}(x')$ and $N_{G[\mathcal{U}^1_{e(i'+2)}]}(y')$, respectively.
Again apply \Cref{lem: find path} to find a path $P''$ connecting a vertex $y''\in Y''$ and a vertex $x''\in X''$ which crosses $(Y'', U^{1}_{(r-1)(i'+2)+1}\setminus~W,\dots,U^{1}_{(r-1)t}, U^1_1,\dots, U^{1}_{(r-1)(i'-2)-1}\setminus~W, X'')$ and is disjoint from $V(F_2)$. This similarly follows as every set in the tuple has size at least $\frac{1}{4}\delta^{4r}n$. By the definitions of $X''$ and $Y''$, there is an edge of $G[\mathcal{U}^1_{e(i'-1)}]-(W\cup V(F_2))$ containing both $x'$ and $x''$, and an edge of $G[\mathcal{U}^{1}_{e(i'+2)}]-(W\cup V(F_2))$ containing both $y'$ and $y''$. These two edges together with $P'$ and $P''$ yield a copy $F_1$ of $F$ that crosses $\mathcal{U}^{1}=(U^1_1\setminus W,\dots, U^1_{(r-1)t}\setminus W)$ and is disjoint from $V(F_2)$.

As each $U^1_j$ is an $(F,\beta n,2^{10\varepsilon^{-2}} )$-closed set, 
\Cref{lem:merge_closed_and_nonclosed_sets} implies that $V_1$ is $(F,\beta n - 10\cdot 2^{10\varepsilon^{-2} }t , 6\cdot 2^{10\varepsilon^{-2}+1 }t)$-closed.
As $1/n \ll \beta \ll  \delta \ll 1/t, \varepsilon \leq 1$, this proves the claim.
\end{proof}

Now we apply \Cref{lem:closed_to_absorbing} to find an absorbing set. One can observe that the condition \ref{closed 2} is easy to deduce from the minimum degree condition and \Cref{lem: find path}, and thus by the choice $\gamma\ll \eta \ll \beta, \delta\leq 1$, $G$ has a $\gamma$-absorbing set $A$ of size at most $\eta n$.
Then the $r$-graph $G'=G-A$ admits a balanced partition $\mathcal{V}'=(V'_1,\dots, V'_{(r-1)t})$ where $V'_i= V_i\setminus A$. 
Assume each $V'_i$ has size $n' \geq (1-\delta)n$.
As $|A|\leq \eta n$ and $\eta \ll \varepsilon$, \ref{eq: M3} implies $\delta(G[\mathcal{V}'_{e}])\geq \tfrac{1}{2}\varepsilon (n')^{r-1}$ for each $e\in E(F)$.
Now we prove the following claim.

\begin{claim}
 There exists a partition $U_{i,1},\dots, U_{i,q}$ of $V'_i$ for each $i\in [(r-1)t]$
 such that the following holds.
\begin{proplist}
\item For each $i\in [(r-1)t]$ and $j\in [q]$, we have $|U_{1,j}|=\dots = |U_{(r-1)t,j}| \geq n'/T^2$. \label{V'U1}
\item  For each $j\in [q]$, there exists $i'=i'(j)\in [t]$ such that the tuples $(U_{i,j}: i\in e(4i'-2))$ and $(U_{i,j}: i\in e(4i'))$ are both $(\gamma^3,\eps/100+)$-regular. \label{V'U2}
\end{proplist}
\end{claim}
\begin{proof}
For each $i'\in [t/10]$, we apply \Cref{lem: reg matching} to $G[\mathcal{V}'_{e(4i'-2)}\cup \mathcal{V}'_{e(4i')}]$ where $(V'_{i'}: i'\in e(4i-2))$, $(V'_{i'}: i'\in e(4i))$, $\eps/10$, $T$, $\gamma^3$, $10n'/t$ playing the roles of $(V_1,\dots,V_r), (V_{r+1},\dots, V_{2r})$, $\varepsilon$, $T$, $\tau$, $n^*$ respectively.
Then we obtain subsets 
$V_{i,1},\dots, V_{i,p_{i'}}$ of $V'_{i}$ for each $i\in e(4i'-2)\cup e(4i')$ satisfying \ref{reg matching 1}--\ref{reg matching 3}.
For each $i'\in [t/10]$ and $j\in [p_{i'}]$, the size of $|V_{i,j}|$ is the same for all $i\in e(4i'-2)\cup e(4i')$, and let $a_{i',j}$ be the common size. Then \ref{reg matching 3} implies 
$$\sum_{i'\in [t/10]}\sum_{j\in [p_{i'}] } a_{i',j} = \sum_{i'\in [t/10]} 10n'/t = n'.$$

Let $q= \sum_{i\in [t/10]} p_i$.
Now for each $i\in [(r-1)t]$, we arbitrarily partition each $V'_{i}$ into sets $U_{i,1},\dots, U_{i,q}$ of size $a_{1,1},\dots, a_{1,p_1}, a_{2,1},\dots, a_{t/10,p_{t/10}}$ in such a way that 
for $i'\in [t/10]$, $i\in e(4i'-2)\cup e(4i')$ and $j\in [p_i]$, we have 
$$ U_{i,q'+j}= V_{i,j} \text{ where } q'= \sum_{i^*< i}p_{i^*}.$$
Then it is straightforward to check that this partition satisfies \ref{V'U1} and \ref{V'U2}.
\end{proof}

Consider a partition as in the above claim. We now fix $j\in [q]$ and the tuple $\mathcal{U}_j = (U_{1,j},\dots, U_{(r-1)t,j})$ and show that $G[\mathcal{U}_j]$ contains $(1-\gamma^2) |U_{1,j}|$ vertex-disjoint $\mathcal{U}_j$-transversal copies of $F$.  Let $i'=i'(j)$ as in \ref{V'U2} and let $n_j:= |U_{1,j}|=\dots = |U_{(r-1)t,j}|$.
To show this, assume that we have chosen
 $s < (1-\gamma^2)|U_{1,j}|$ vertex-disjoint $\mathcal{U}_j$-transversal copies of $F$ in $G[\mathcal{U}_j]$. Delete the vertices in those copies of $F$, then for each $j\in [(r-1)t]$, we are left with vertex sets $U'_i\subseteq U_{i,j}$ of the size at least $n_j- s \geq \gamma^2 n_j$. 
By \ref{V'U1}, we have $\gamma^2 n_j > n'/T^3 > \alpha n> \alpha^*_{\mathcal{U}_{e(4i'-1),j}}(G)$ where $\mathcal{U}'_{e(4i'-1)}= (U'_{i}: i\in e(4i'-1))$, there exists a matching $M$ in $G[\mathcal{U}'_{e(4i'-1)}]$ of size at least $ n_j -s - \alpha n \geq \frac{1}{2}\gamma^2 n_j$.

We claim that $M$ contains an edge $f$ that connects a vertex $x\in U'_{(r-1)(4i'-2)}$ and a vertex $y\in U'_{(r-1)(4i'-1)}$
such that each of $x,y$ has degree at least $ \frac{1}{2}\delta (n_j-s)^{r-1}$ in $G[\mathcal{U}'_{e(4i'-2)}]$ and $G[\mathcal{U}'_{e(4i')}]$, respectively.
Indeed, if it is not true, at least half of the vertices in $V(M)\cap U'_{(r-1)j}$ have degree less than $ \frac{1}{2}\delta (n_j-s)^{r-1}$ in $G[\mathcal{U}'_{e(k)}]$ for some $k\in \{4i'-2, 4i'\}$, then the set of these vertices together with the rest of the sets $U'_{i}$ with $i\in e(k)\setminus \{k(r-1)\}$ form a tuple with density less than $\frac{1}{2}\delta < \eps/100-\gamma^3$, a contradiction to the $(\gamma^3,\eps/100+)$-regularity of the tuple $\mathcal{U}_{e(k),j}$.

Thus, for such choice of $f$ and $x,y$, we can apply \Cref{lem: find a set} to find subsets $X^*\subseteq U'_{(r-1)(4i'-3)}$ and $Y^*\subseteq U'_{(r-1)4i'}$ such that every vertex $x^*\in X^*$ is in some edge of $G[\mathcal{U}'_{e(4i'-2)}]$  containing $x$ and every vertex $y^*\in Y^*$ is in some edge of $G[\mathcal{U}'_{e(4i')}]$ containing $y$. Indeed, \Cref{lem: find a set} implies that $|X^*|,|Y^*|\ge \delta\gamma^2|U'_{(r-1)(4i'-2)}| > n/T^3$.
As $n/T^3 > 2(r-1)t \alpha n$, we can apply \Cref{lem: find path} to find a path connecting some vertex $x^*\in X^*$ and some vertex $y^*\in Y^*$ that crosses $(Y^*, U'_{(r-1)(4i')+1},\dots,U'_{(r-1)t}, U'_{1}, \dots, U'_{(r-1)(4i'-3)-1}, X^*)$. This together with the edge $f$ and the aforementioned edge containing both $x,x^*$ and the edge containing both $y,y^*$ altogether form a cycle in $G[\mathcal{U}'_j]$ that is disjoint from the previous $s$ copies. This shows that we can find $(1-\gamma^2)n_j$ vertex-disjoint $\mathcal{U}_j$-transversal copies of $F$ in $G[\mathcal{U}_j]$.
Taking the union of those copies for all $j\in [q]$, we obtain a family $\mathcal{F}_1$ of vertex-disjoint $\mathcal{V}$-transversal copies of $F$ in $G'$ which covers all but less than $\gamma^2 n' \leq \gamma n$ vertices. Let $U$ be the set of the uncovered vertices. Then, as $A$ is a $\gamma$-absorbing set, $A\cup U$ admits an $F$-factor, and this together with $\mathcal{F}_1$ yields a desired $F$-factor.

\section*{Acknowledgement}
We would like to thank Hyunwoo Lee for the valuable discussions.

\appendix

\appendix

\section{Proof of \Cref{lem:general_concent}}
In this section, we consider weighted graphs. A weighted $r$-graph $G$ is an $r$-uniform hypergraph equipped with a weight function $w:E(G)\rightarrow \mathbb{R}_{\geq 0}$. We say $w(e)$ is the weight of the edge $e$ of $G$. We sometimes consider non-edges as edges with weight zero. 
Many concepts in hypergraphs naturally extend to weighted hypergraphs. For example, the degree $d_G(U)$ of a vertex set $U$ is the weight sum of all edges containing $U$. Other notations like $d_{G}(v,\mathcal{U})$ also naturally extend by replacing the number of edges with weight sum.
The weight of the graph, denoted by $w(G)$, is the sum of weights over all edges in $G$. Also denote by $w_G(Y_1, \ldots, Y_r)$ the total weight of edges crossing $(Y_1,\cdots, Y_r)$.
Our goal in this section is to prove the following lemma which immediately implies \Cref{lem:general_concent}.
\begin{lemma}\label{lem:general_concent_weighted}
    Suppose $0<1/n \ll 1/t, 1/r, p_0, \eta \leq 1$.
    Let $G$ be an $n$-vertex weighted $r$-graph with $w(G) \geq \varepsilon n^r$ edges and all the edges $e$ have weight $w(e)\in[0,1]$.
    Let $p_1,\dots, p_t$ be positive numbers with $p_i\geq p_0$ for all $i\in [t]$ and $\sum_{i\in [t]} p_i=1$.
        Let $\{X_1, \ldots, X_t\}$ be a random partition of $V(G)$, each of size $p_in$, chosen uniformly at random.
    Then with probability $1-o(n^{-1})$, the following holds.
    Let $Y_1,\dots, Y_{r} \in \{X_1,\dots, X_t\}$ with possible repetitions. 
    Then  we have
    $$w_G(Y_1,\dots, Y_{r}) = (1 \pm \eta) r! w(G) \prod_{i\in [r]} \frac{|Y_i|}{n}.$$
\end{lemma}

Indeed, \Cref{lem:general_concent} follows by applying \Cref{lem:general_concent_weighted} to the hypergraph $H$ with the edge set $N_G(v)$ where $\delta_1(G)\geq \varepsilon n^{r-1}$ and taking a union bound over all $v\in V(G)$.

The following well-known equation is also useful for us later.
\begin{equation}\label{eq: binomial}
    \sum_{\ell' \geq \ell} {r \choose \ell'}{\ell' \choose \ell} p^{\ell'}(1-p)^{r-\ell'} = {r \choose \ell} p^{\ell} \text{ for any } \ell\in [r] \text{ and } p\in [0,1]. 
\end{equation}


We also need the following Kim--Vu polynomial concentration.
\begin{lemma}[Kim--Vu polynomial concentration~\cite{Kim2000concentration}]\label{lem:KimVu}
    Let $\mathcal{G}$ be an $n$-vertex hypergraph such that all edges have size at most $r$ and $(X_v)_{v \in V(\mathcal{G})}$ be a set of mutually independent Bernoulli random variables.
    Let $Y_{\mathcal{G}} = \sum_{e \in \mathcal{G}} w(e)\prod_{v \in e} X_v$ and for each $A \subseteq V(\mathcal{G})$, let $Y_{\mathcal{G}_A} = \sum_{e \in \mathcal{G}, A \subseteq e} w(e)\prod_{v \in e \setminus A} X_v$.
    Let $\mathcal{E}_i = \max_{|A|=i} \mathbb{E}[Y_{\mathcal{G}_A}]$ for each $0 \leq i \leq r$ and write $\mathcal{E} = \max_{0 \leq i \leq r} \mathcal{E}_i$ and $\mathcal{E}' = \max_{1 \leq i \leq r} \mathcal{E}_i$.
    Then for any $\lambda>1$, 
    $$\mathbb{P}\left[ |Y_{\mathcal{G}} - \mathbb{E}[Y_\mathcal{G}]| > 8^r  \sqrt{r!\mathcal{E}\mathcal{E'}} \lambda^r \right] \leq 2e^{2-\lambda}n^{r-1}.$$
\end{lemma}

Note that \Cref{lem:KimVu} is regarding a random vertex set where each vertex is chosen independently at random. On the other hand, \Cref{lem:general_concent_weighted} concerns the case where a set of fixed size is chosen uniformly at random. 
As one may expect, there is a close relationship between these two models and here we start with a observation as follows.

\begin{lemma}\label{lem:hypergeometric}
    Let $0<1/n \ll p_0$, $p \in [p_0, 1-p_0]$ be a constant,  and $m=\lceil pn \rceil$. 
    For an $n$-vertex hypergraph $G$, let $\mathcal{X}$ be a family of subsets of $V(G)$ that is closed under taking subsets (supersets, respectively).
    Let $V(p)$ be a random subset of $V(G)$ chosen by selecting each vertex with probability $p$ independently at random and $V(m)$ be a uniform randomly chosen $m$-element subset of $V(G)$.
    Then we have $$ \mathbb{P}(V(m) \in \mathcal{X}) \leq 3\mathbb{P}(V(p) \in \mathcal{X}).$$
\end{lemma}
\begin{proof}
    We only prove when $\mathcal{X}$ is closed under subset relation. 
    The other case is similar.
    By the monotonicity of $\mathcal{X}$, we have $\mathbb{P}(V(m) \in \mathcal{X}) \geq \mathbb{P}(V(m') \in \mathcal{X})$ when $m' \geq m$. 
    Thus, 
    \begin{align*}
        \mathbb{P}(V(p) \in \mathcal{X}) & = \sum_{i=0}^{n} \mathbb{P}(V(p) \in \mathcal{X} \mid |V(p)|=i)\mathbb{P}(|V(p)|=i) \\
        & =  \sum_{i=0}^{n} \mathbb{P}(V(i) \in \mathcal{X})\mathbb{P}(|V(p)|=i) \\
        & \geq \mathbb{P}(V(m) \in \mathcal{X})\sum_{i=0}^{m} \mathbb{P}(|V(p)|=i)\\
        & = \mathbb{P}(V(m) \in \mathcal{X}) \mathbb{P}(\mathrm{bin}(n, p) \leq m)  \\
        & \geq \frac{1}{3} \mathbb{P}(V(m) \in \mathcal{X})
    \end{align*}
    as $n$ is large enough.
\end{proof}

The first step for proving \Cref{lem:general_concent_weighted} is to prove the case when all the $Y_i$ are the same part.
We follow the proof in Appendix A of~\cite{Kang2024perfect} to prove the following concentration result using \Cref{lem:KimVu}.
\begin{lemma}\label{lem:number_of_edge_concentration}
    Let $p, \varepsilon, \eta \in (0, 1)$ be fixed constants and $0<\frac{1}{n} \ll 1/r, p, \varepsilon, \eta$.
    Let $G$ be a weighted $n$-vertex $r$-graph with $w(G) \geq \varepsilon n^r$ and all the edges have weight at most $1$.
    Let $X \subseteq V(G)$ be a random subset of $V(G)$ of size $pn$ chosen uniformly at random. 
    Then with probability at least $1- n^{-3}$, the total weight of edges in $G[X]$ is $(1 \pm \eta)p^r w(G)$.
\end{lemma}
\begin{proof}
    We bound the probability that $w(G[X])$ is at most $(1-\eta)p^rw(G)$ and the other direction separately.
    Since the event that $w(G[X])<(1-\eta)p^rw(G)$ is monotone in the sense that any subset of $X$ satisfies the same, by \Cref{lem:hypergeometric}, it suffice to prove the statement for a $p$-random subset $X$.
    
    For each vertex $v$, let $X_v$ be the indicator random variable for the event $v \in X$ and let $Y_{G} = \sum_{e \in E(G)} w(e) \prod_{v \in e} X_v$ as in \Cref{lem:KimVu}.
    Then by the construction, $Y_{G} = w(G[X])$, and $\mathbb{E}[Y_{G}] = p^{r}w(G) \geq \varepsilon p^{r} n^{r}$.

    Also, for every $i \in [r-1]$ and a subset $A = \{v_1, \ldots, v_i\} \subseteq V(G)$ of size $i$, we have $$\mathbb{E}[Y_{G_A}] = \mathbb{E}[\frac{1}{(r-i)!}d_G(\{v_1\}, \ldots, \{v_i\}, X, \ldots, X)] \leq p^{r-i} n^{r-i}\leq \mathbb{E}[Y_{\mathcal{G}}] \cdot \frac{1}{\varepsilon p n}$$ 
    because $e(\{v_1\}, \ldots, \{v_i\}, V(G), \ldots, V(G))/(r-i)! \leq n^{r-i}$ as all edges of $G$ have weight at most one.

    Thus, $\mathcal{E} = \mathbb{E}[Y_{\mathcal{G}}]$ and $\mathcal{E}' \leq \mathbb{E}[Y_{\mathcal{G}}] \cdot \frac{1}{\varepsilon p n}$.
    By taking $\lambda$ as follows $$\lambda := \left( \frac{(1-\eta)\mathbb{E}[Y_{\mathcal{G}}]}{ 8^{r} \cdot \sqrt{r! \mathcal{E} \mathcal{E}'}} \right)^{1/r} \geq \left(\frac{(1-\eta)\sqrt{\varepsilon p n}}{8^{r} \cdot \sqrt{ r!}}\right)^{1/r} \geq n^{1/(2r+1)},$$
    \Cref{lem:KimVu} implies  
    \begin{align*}
        \mathbb{P}(Y_\mathcal{G} < (1-\eta) p^{r} w(G)) \leq \mathbb{P}\left(|\mathbb{E}[Y_\mathcal{G}] - Y_\mathcal{G}| >  8^{r}  \sqrt{r!\mathcal{E}\mathcal{E'}} \lambda^{r} \right) \leq 2e^2e^{-n^{1/(2r+1)}} n^{r-1} = o(n^{-3}).
    \end{align*}
    It concludes the proof of one direction.

    To show the opposite direction, the collection of $X$ with $w(G[X]) > (1+\eta) p^{r} e(G)$ is closed under superset, one can use \Cref{lem:hypergeometric} again and apply the Kim-Vu polynomial concentration with the same way.
\end{proof}

The next lemma shows the case of $t=2$ of \Cref{lem:general_concent_weighted}.
\begin{lemma}\label{lem:two_sets_concentration}
    Let $p, \varepsilon, \eta \in (0, 1)$ be fixed constants and $n$ is sufficiently large.
    Let $G$ be a weighted $n$-vertex $r$-graph with all the edge weight at most $1$ and $G$ has total weight at least $\varepsilon n^r$.
    Let $X \subseteq V(G)$ be a random subset of $G$ of size $pn$ chosen uniformly at random.
    Then with probability at least $1- 2^{r} n^{-3}$, for every $0 \leq \ell \leq r$, the total weight of edges $e \in E(G)$ such that $|e \cap X|= \ell$ is $(1 \pm \eta) {r \choose \ell} p^{\ell}(1-p)^{r-\ell}w(G)$.
\end{lemma}
\begin{proof}
    For each $0\leq i\leq r$, let $\eta_{i} = \eta/(3r)^{i}$ and $m_{i}$ be the total weight of edges $e \in E(G)$ such that $|e \cap X|=i$.

We apply induction on $i$ to prove that
\begin{align}\label{eq: mri}
    m_{r-i} = (1 \pm \eta_{r-i} ) {\binom{r}{i}} p^{r-i}(1-p)^{i}w(G)
\end{align} holds with probability at least $1- 2^{i} n^{-3}$.
When $i = 0$, this follows from  \Cref{lem:number_of_edge_concentration}.

Suppose that the statement holds for $0\leq i\leq r-\ell-1$, namely $m_r,m_{r-1},\cdots, m_{r+1}$ satisfy \eqref{eq: mri}. Now we shall prove that $m_{\ell}$ satifies \eqref{eq: mri}.
 We define an auxiliary weighted $\ell$-graph $G'$ on the vertex set $V(G)$ where each $\ell$-set $f$ has weight $\frac{d_G(f)}{n^{r-\ell}}$.
 Then we have  $w(G') = \binom{r}{\ell} w(G)/n^{r-\ell}$.
    By \Cref{lem:number_of_edge_concentration}, with probability at least $1- n^{-3}$, the weight sum in $G'[X]$ is $(1 \pm \eta_\ell/3) \binom{r}{\ell} p^{\ell} w(G)/n^{r-\ell}$.

   We now consider the contribution of each edge $e \in E(G)$ with $|e \cap X| = \ell'$ in $G'[X]$ for each $0 \leq \ell' \leq r$.
    If $\ell' < \ell$, then it is not counted in $G'[X]$ and if $\ell' \geq \ell$, then it contributes ${\ell' \choose \ell}w(e)/n^{r-\ell}$ in $G'[X]$.
    Thus 
    $$w(G'[X]) =\sum_{\ell' =\ell}^r {\ell' \choose \ell} m_{\ell'}/n^{r-\ell}.$$
As $w(G'[X])= (1 \pm \eta_\ell/3) \binom{r}{\ell} p^{\ell} w(G)/n^{r-\ell}$, this implies that
    $$ \sum_{\ell' =\ell}^r {\ell' \choose \ell} m_{\ell'} = (1 \pm \eta_\ell/3) {r \choose \ell} p^{\ell} w(G)$$
    holds with probability at least $1- n^{-3}$.

By the induction hypothesis, with probability at least $1 - (2^0+2^1+\dots + 2^{r-\ell-1}) n^{-3} = 1 - (2^{r-\ell}-1)n^{-3}$, \eqref{eq: mri} holds for all $0\leq i\leq r-\ell -1$.
Thus a union bound yields that, with probability at least 
$1- 2^{r-\ell} n^{-3}$, we have
    $$ m_{\ell} = (1 \pm \eta_\ell/3) {r \choose \ell}p^{\ell} w(G) - \sum_{\ell' > \ell} (1 \pm \eta_{\ell'}) {r \choose \ell'}{\ell' \choose \ell} p^{\ell'}(1-p)^{r-\ell'} w(G). $$
We simplify this using \eqref{eq: binomial}, and obtain 
$m_{\ell} = (1\pm \eta_\ell) \binom{r}{\ell} p^{\ell} (1-p)^{r-\ell} w(G)$. This proves \eqref{eq: mri}, thus also proves the lemma.
\end{proof}
We now prove \Cref{lem:general_concent_weighted}.
\begin{proof}[Proof of \Cref{lem:general_concent_weighted}]
    We apply induction on $t$ to show that the statement holds with probability at least $1- t^{rt} n^{-2}$.
    The base case $t=2$ is derived from \Cref{lem:two_sets_concentration}.

Fix a choice $Y_1,\dots, Y_r\in \{X_1,\dots, X_t\}$.
Without loss of generality, assume that $Y_1 = \dots = Y_{\ell}=X_1$ and
$Y_{\ell+1},\dots, Y_r \in \{X_2,\dots, X_t\}$ for some $0\leq \ell\leq r$.

We first select a set $X_1$ of size $p_1 n$ uniformly at random and then select a random partition $X_2, \ldots, X_t$ of $V(G) \setminus X_1$ with the appropriate sizes.
    By \Cref{lem:two_sets_concentration}, with probability $1-o(n^{-2})$, the total weight of edges $e \in E(G)$ such that $|e \cap X_1| = \ell$ is $(1 \pm \eta/3) {r \choose \ell} p_1^{\ell}(1-p_1)^{r-\ell}w(G)$.
    We now define an auxilary weighted $(r-\ell)$-graph $H$ on $V(G) \setminus X_1$ by letting the weight of any $(r-\ell)$-set $f \subseteq V(G) \setminus X_1$ be $\frac{\sum_{e\setminus X_1=f}w(e)}{n^{\ell}}$.
     Then, with probability $1-o(n^{-2})$, we have 
    $$w(H) = (1 \pm \eta/3) {r \choose \ell} p_1^{\ell}(1-p_1)^{r-\ell}w(G)/n^{\ell}.$$

    By the induction hypothesis, with probability at least $1- (t-\ell)^{r(t-\ell)} n^{-2} \geq 1- (t-1)^{r(t-1)} n^{-2}$, $$w_H(Y_{\ell+1},\dots, Y_{r}) = (1 \pm \eta/3) (r-\ell)! w(H) \prod_{\ell+1 \leq i \leq r} \frac{|Y_i|}{n-p_1 n}.$$
    Since $w_G(Y_1, \ldots, Y_r) = w_H(Y_{\ell+1}, \ldots, Y_r) \times n^{\ell} \times \ell!$ where the last $\ell!$ comes from the ordering of $\ell$ elements in $X_1$, we have 
    $$w_G(Y_1,\dots, Y_{r}) = (1 \pm \eta) r! w(G) \prod_{i\in [r]} \frac{|Y_i|}{n}.$$
    By taking union bound for all $t^r$ choices of $Y_i$, we conclude that the statement holds with probability at least 
    $1 - t^{r}((t-1)^{r(t-1)} \cdot n^{-2} ) \geq  1- t^{rt} n^{-2}$. This establishes the induction, hence the lemma as well.
\end{proof}

\end{document}